\documentclass{article}
\usepackage{amsmath}
\usepackage{amssymb}
\usepackage{xcolor}
\usepackage{amsthm}
\usepackage[margin=0.8in]{geometry}
\usepackage{graphicx}
\usepackage{bbm}
\usepackage{authblk}
\RequirePackage{subcaption}

\title{A Fully Discrete Nonnegativity-Preserving FEM for a Stochastic Heat Equation}

\date{}

\newtheorem{definition}{Definition}[section]
\newtheorem{theorem}{Theorem}[section]
\newtheorem{lemma}{Lemma}[section]
\newtheorem{proposition}{Proposition}[section]
\newtheorem{assumption}{Assumption}[section]
\newtheorem{corollary}{Corollary}[section]

\usepackage[shortlabels]{enumitem}

\begin{document}

\author[1]{Owen Hearder}
\author[2]{Claude Le Bris}
\author[3]{Ana Djurdjevac}

\affil[1]{Freie Universität Berlin, Arnimallee 9, 14195 Berlin, Germany, owen.hearder@fu-berlin.de}
\affil[2]{École des Ponts and INRIA,  MATHERIALS project-team,  6-8 avenue Blaise Pascal, 77455 Marne La Vallée, France, claude.le-bris@enpc.fr}
\affil[3]{ University of Oxford, Mathematical Institute, Oxford, UK, ana.djurdjevac@maths.ox.ac.uk and Freie Universität Berlin, Arnimallee 6, 14195 Berlin, Germany, adjurdjevac@zedat.fu-berlin.de}

\maketitle

\begin{abstract}

We consider a stochastic heat equation with nonlinear finite-rank space-coloured multiplicative noise that admits a unique nonnegative solution when given nonnegative initial data. Inspired by existing results for fully discrete finite difference schemes and building on the convergence analysis of semi-discrete mass-lumped finite element approximations, a fully discrete numerical method is introduced that combines mass-lumped finite elements with a Lie-Trotter splitting strategy. This discretization preserves nonnegativity at the discrete level and is shown to be convergent under suitable regularity conditions. A rigorous convergence analysis is provided, highlighting the role of mass lumping in ensuring nonnegativity and of operator splitting in decoupling the deterministic and stochastic dynamics. Numerical experiments are presented to confirm the convergence rates and the preservation of nonnegativity. In addition, we examine several numerical examples outside the scope of the established theory, aiming to explore the range of applicability and potential limitations of the proposed method.

\end{abstract}

\section{Introduction}

The development of nonnegativity-preserving numerical schemes for stochastic partial differential equations (SPDEs) is motivated by the fact that many SPDEs describe the evolution of particle densities or other intrinsically nonnegative physical quantities; see, for example, \cite{carrillo2025positivity, almgren, DjurdKrempPerk, MAGALETTI2022111248}. Ensuring that a numerical discretization retains this structural property is essential for maintaining physical consistency and for achieving stable long-time simulations.\par

As a first step toward more complex conservative SPDE models, the focus of this work is a stochastic heat equation on a bounded spatial domain $\mathcal{D}\subset\mathbb{R}^d$, driven by nonlinear multiplicative noise that is white in time and coloured in space with finite-rank spatial covariance, of the form

\begin{equation}\label{intro_equation}
\begin{split}
    \text{d}u(t)&=\Delta u(t)\text{d}t+ f(u(t))\sum_{k=1}^K e_k \text{d}B^k_t,\quad\text{for}\quad t\in(0,T],\\
    u(0)&=u_0,
\end{split}
\end{equation}

where $(e_k)_{k=1}^K$ is an $L^2(\mathcal{D})$-orthonormal family of continuous functions on $\overline{\mathcal{D}}$ and $(B^k)_{k=1}^K$ is a family of independent standard real-valued Brownian motions. Under suitable regularity assumptions, the nonnegativity of solutions to \eqref{intro_equation} for nonnegative initial data  and nonlinearity satisfying  $f(0)=0$ has been established in \cite{cresson}. In settings with weaker regularity, nonnegativity preservation can still be guaranteed by comparison principles, such as those in \cite[Theorem~2.1]{donati1993white} and \cite[Theorem~2.5]{Kotelenez-1992}.

To the best of our knowledge, the analysis of nonnegativity preservation for fully discrete numerical schemes for SPDEs has so far been confined to finite difference-based discretizations. For instance, \cite{Yang_Yang_Zhang} studies a generalization of \eqref{intro_equation} and establishes stability properties of a fully discrete finite difference method. In a related direction, \cite{BCU_1} introduces a Lie-Trotter splitting scheme that preserves nonnegativity for an SPDE driven by space-time white noise, thereby restricting the analysis to one spatial dimension. Additional references can be found in \cite{BCU_1}.

The use of finite element methods in this context offers several advantages over finite difference approaches. Finite elements naturally align with variational formulations and bring with them a rich numerical analysis framework. Moreover, they accommodate complex geometries, support adaptive mesh refinement, and can be easily implemented using standard finite element software developed for deterministic parabolic equations. These features make finite element methods an appealing choice for extending nonnegativity-preserving discretizations to more general SPDEs.

Nonnegativity-preserving finite element schemes have been extensively studied in the deterministic parabolic setting. Early results were obtained in \cite{fujii}, where a lumped mass method was analysed and geometric conditions on the underlying mesh were identified to guarantee nonnegativity. Subsequent refinements of the lumped mass approach appeared in \cite{Chatzipantelidis_Horvath_Thomee} and \cite{thomeepaper}, which further clarified the role of mesh structure and its interaction with the discrete maximum principle. The extension of these ideas to the stochastic setting was developed in \cite{Djurdjevac}, where the same equation \eqref{intro_equation} was considered and nonnegativity preservation, as well as strong convergence estimates in the $H^1$-norm, were established for a semi-discrete spatial discretization.

The numerical scheme studied in this work, denoted by $(U_m^{LT,h})_{m=0}^M$, combines the mass-lumped finite element discretization of \cite{Djurdjevac} with the Lie-Trotter splitting scheme introduced in \cite{BCU_1}. This yields a fully discrete approximation of \eqref{intro_equation}. The key idea behind the splitting approach is to decompose the equation (\ref{intro_equation}) into easily solved subsystems, in our case a purely stochastic subsystem and a purely deterministic subsystem, that can be solved explicitly. The scheme is given by
\begin{align*}
\hat{U}^{LT,h}_m&=\exp\left(\sum_{k=1}^K\delta B^k_{m}D_{g\left(u^{LT,h}_m\right)} D_{e_k} - \frac{\tau}{2}\sum_{k=1}^K D_{g\left(u^{LT,h}_m\right)}^2 D_{e_k}^2\right)U^{LT,h}_m\\
U^{LT,h}_{m+1}&=\exp\left(-\tau\mathcal{M}_L^{-1}\mathcal{S}\right)\hat{U}^{LT,h}_m,
\end{align*}
where $\tau=T/M$ denotes the time step size, the $\delta B^k_{m}$ are increments of the driving Brownian motions $B^k$, the $D_{e_k}$ are diagonal matrices containing nodal values of the $e_k$, and $D_{g\left(u^{LT,h}_m\right)}$ is the diagonal matrix of nodal values of $g(u^{LT,h}_m)$ where $g(s)=f(s)/s$. The matrices $\mathcal{M}_L$ and $\mathcal{S}$ encode the spatial discretization and denote, respectively, the lumped mass and stiffness matrices associated with the mass lumped finite element method, which will be defined later on in Section \ref{section3}.\par

The main contributions of this work are as follows:
\begin{itemize}
\item We extend one of the two fully discrete schemes proposed in
\cite{Djurdjevac} from the linear case
\( f(u)=\lambda u \) to general Lipschitz nonlinearities. We prove that the resulting method
preserves nonnegativity (Theorem~\ref{splitting_positivity}). In
addition, we derive uniform second-moment bounds for the $L^2$ norm of the
discrete solution, providing key stability estimates for the subsequent
convergence analysis (Theorem~\ref{moment_bounds}).

    \item The strong $L^2$-convergence of the fully discrete numerical scheme 
    with temporal rate $\frac{1}{2}-\varepsilon$ (for arbitrary small $\varepsilon$) and spatial rate $1$ (see Theorem \ref{convergence} and Corollary \ref{main_result}).
    \item Numerical experiments, performed on a two-dimensional spatial domain,  which illustrate the strong $L^2$-convergence temporal rate and an improved spatial rate. Further numerical experiments are presented to investigate the performance of the scheme in regimes not addressed by our theoretical results, offering insight into its potential applicability beyond the proven setting (see Section $6$).
\end{itemize}

The findings of this work represent an initial stage in a broader research program and are not intended as definitive conclusions. We offer a snapshot of current ideas, while a complete theoretical framework is left for future work. In this spirit, we establish the convergence and nonnegativity proof of the fully discrete scheme as a foundation for the development of more advanced methods and more complex equations. The numerical experiments are designed in the same exploratory spirit: they investigate several potential future directions and illustrate how the scheme behaves in settings that lie beyond the scope of our current theoretical results.

In particular, the numerical studies demonstrate that the fully discrete scheme is feasible in practice and achieves, at a minimum, the expected convergence rates. They also reveal certain limitations of the method. For instance, when the number of Brownian motions becomes large, the scheme no longer exhibits the desired behaviour, reflecting the dependence of key constants on the number of noise terms. 
Regarding spatial convergence, the experiments clearly indicate that the proven rate is not optimal. This is consistent with expectations from deterministic finite element theory, but the adaptation of deterministic methods to the stochastic setting, in particular their combination with stochastic calculus, is not straightforward. Addressing this issue is therefore left for future investigation.

Motivated by the importance of weak convergence for SPDEs (\cite{GeissertMatthiasKovacs, debussche2009weak, BrehierGoudenege, Kruse})—where expectations of functionals are often the primary observables—we further investigate the weak error through numerical experiments. While these indicate convergence and a substantially smaller error compared to the strong error, no reliable conclusions about the weak convergence rate can be drawn at this stage without employing additional variance-reduction techniques or performing significantly longer simulations.

The paper is organized as follows. We begin in Section \ref{section2} with the analytical background for the continuous problem \eqref{intro_equation}, establishing well-posedness and nonnegativity. Section \ref{section3} then introduces the semi-discrete spatial discretization from \cite{Djurdjevac}, together with strong error estimates and auxiliary results for the spatial scheme. Building on this, Section \ref{section4} develops the fully discrete Lie-Trotter splitting method and proves its nonnegativity. Strong convergence rates of the Lie-Trotter scheme towards the semi-discrete solution are derived in Section \ref{section5}.
Section \ref{section6} presents numerical experiments validating the convergence rates and preservation of nonnegativity, together with examples beyond the theoretical setting.
The code used for all numerical experiments is publicly available, and a link is provided in Section \ref{section6}.
\section{Preliminaries on the Continuous Problem}\label{section2}

As already mentioned in the introduction, the equation we consider is the heat equation with finite-rank space-coloured multiplicative noise:
\begin{equation}\label{main_equation}
    \begin{split}
    \text{d}u(t)&=\Delta u(t)\text{d}t+f(u(t))\sum_{k=1}^K e_k \text{d}B^k_t,\quad\text{for}\quad t\in(0,T],\\
    u(0)&=u_0,\\
    \end{split}
\end{equation}
with homogeneous Dirichlet boundary conditions on a spatial domain $\mathcal{D}\subset\mathbb{R}^d$. The finite-rank assumption is imposed for simplicity; we expect that the analysis can be extended to trace-class noise with suitably decaying coefficients, at the cost of an additional noise-truncation error. We first state assumptions on equation (\ref{main_equation}) and then define the notion of solution we consider.
\begin{assumption}\label{assumption_1}
The following assumptions on equation (\ref{main_equation}) are made:
\begin{enumerate}[(a)]
    \item the time domain is finite, $T\in(0,\infty)$, and the spatial domain, $\mathcal{D}\subset\mathbb{R}^d$, for $d=1,2$ or $3$, is either a non-empty, bounded and open polyhedral domain or a non-empty, bounded, open and convex domain with $C^2$ boundary;
    \item the Laplacian is defined with homogeneous Dirichlet boundary conditions, namely $\Delta:H^1_0(\mathcal{D})\cap H^2(\mathcal{D})\subset L^2(\mathcal{D})\to L^2(\mathcal{D})$;
    \item for the noise, we have $K\in\mathbb{N}$, and $B^k$, for $1\leq k\leq K$, are one-dimensional independent Brownian motions with respect to a complete probability space with normal filtration, $(\Omega,\mathcal{F},(\mathcal{F}_t)_{t\in[0,T]},\mathbb{P})$, and the $e_k$ are pairwise $L^2(\mathcal{D})$-orthonormal and $e_k \in C(\overline{\mathcal{D}})$;
    \item the nonlinearity $f:\mathbb{R}\to\mathbb{R}$ is globally Lipschitz, with Lipschitz constant $L_f\geq0$, of class $C^1(\mathbb{R})$ and $f(0)=0$;
    \item the initial condition satisfies $u_0\in H^2(\mathcal{D})\cap C_0(\overline{\mathcal{D}})$ and $u_0\geq0$.
\end{enumerate}
\end{assumption}
We note that the conditions on the spatial domain are so that the usual interpolation error bounds hold true.\par
We also define the constant
\[
K_e:=\left(\sum_{k=1}^K||e_k||_{L^\infty(\mathcal{D})}^2\right)^\frac{1}{2}.
\]
Note that under Assumption \ref{assumption_1}, the Laplace is the generator of a contraction semigroup $(E(t))_{t\geq0}$ on $L^2(\mathcal{D})$.\par
Since we want to utilize semi-group bounds, we will consider mild formulation of the solution. For context, the notion of a predictable process is a measurable process with respect to the predictable $\sigma$-algebra on $[0,T]\times\Omega$. An exact definition can be found in \cite{DaPrato_Zabczyk}.

\begin{definition}\label{mild_solution}
    Under Assumption \ref{assumption_1}, a predictable stochastic process $u:[0,T]\times\Omega\to L^2(\mathcal{D})$ is called a mild solution to (\ref{main_equation}) if
    \[
        u(t)=E(t)u_0 + \sum_{k=1}^K\int_0^t E(t-s)(f(u(s))e_k)\text{d}B^k_s\quad\quad\mathbb{P}-\text{a.s.}
    \]
    for each $t\in[0,T]$, where $(E(t))_{t\in[0,T]}$ is the semigroup on $L^2(\mathcal{D})$ generated by the Laplacian.
\end{definition}

We note there is also the notion of variational solution (see \cite[Definition 5.1.2 ]{Liu_Rockner}), which is used in the paper \cite{Djurdjevac} in an equivalent form. In the case of equation (\ref{main_equation}) under Assumption \ref{assumption_1}, these two notions of solution are equivalent. This is due to the fact that the $e_k$ are uniformly bounded over $\overline{\mathcal{D}}$ and there are a finite number of them. An explanation of the equivalence is contained in \cite[Appendix G]{Liu_Rockner} or \cite[Section 6]{DaPrato_Zabczyk}.\par
From \cite[Theorem 5.1.3]{Liu_Rockner} or \cite[Theorem 2.25]{Kruse}  for well-posedness and \cite[Theorem 2.1]{donati1993white} for the comparison principle, we have the following result.

\begin{theorem}
    Under Assumption \ref{assumption_1}, equation (\ref{main_equation}) has a unique solution in the sense of Definition \ref{mild_solution}. Moreover
    \[
        \mathbb{E}\left[\sup_{t\in[0,T]}||u(t)||^2_{L^2(\mathcal{D})}\right]<\infty
    \]
    and 
    \[
        \mathbb{P}\bigg(u(t,x)\geq0\quad\text{for every}\:\: t\in[0,T]\:\:\text{and a.e.}\:\: x\in\mathcal{D}\bigg)=1.
    \]
\end{theorem}

\section{The Mass-Lumped Discretization}\label{section3}

We now introduce the mass-lumped discretization which was introduced in the SPDE setting by \cite{Djurdjevac}.

\subsection{Triangulation and Finite Element Space}

If $\mathcal{D}\subset \mathbb{R}^d$ is a non-empty, bounded, open polyhedral domain, 
a triangulation $\mathcal{T}=\{\mathcal{C}\}_{\mathcal{C}\in\mathcal{T}}$ is a partition of $\overline{\mathcal{D}}$ 
into $d$-dimensional simplices $\mathcal{C}$ with mutually disjoint interiors.  
For a non-empty, bounded, convex domain $\mathcal{D}\subset\mathbb{R}^d$ with $C^2$ boundary, 
a triangulation is defined analogously as a partition of $\overline{\mathcal{D}_{\mathcal{T}}}$ 
into simplices $\mathcal{C}$, where $\mathcal{D}_{\mathcal{T}}\subset\mathcal{D}$ is a polyhedral 
subdomain whose boundary vertices lie on $\partial\mathcal{D}$.

An admissible triangulation is one where every vertex lies only on the vertices of other simplices or on the boundary, $\partial\mathcal{D}$. A triangulation is denoted with a subscript, $\mathcal{T}_h$, to indicate the maximum diameter of the simplices, $h=\max_{\mathcal{C}\in\mathcal{T}}\text{diam}(\mathcal{C})$. In the case of a non-polyhedral domain, we then write $\mathcal{D}_h:=\mathcal{D}_{\mathcal{T}}\subset\mathcal{D}$.\par

For a given triangulation $\mathcal{T}_h$, let $S_h$ denote the continuous functions on $\mathcal{D}$ which vanish on the boundary, are linear on each $\mathcal{C}\in\mathcal{T}_h$ and are zero on $\mathcal{D}\setminus\mathcal{D}_h$. A function in $S_h$ is uniquely determined by its values on the vertices of $\mathcal{T}_h$ which lie in the interior of $\mathcal{D}$. We define $n_h$ to be the number of interior vertices. With an enumeration of the interior vertices, $\{P_j\}_{j=1}^{n_h}$, we may define the basis functions $\Phi_j$, which take the value of $1$ at $P_j$ and zero at every other interior vertex. We then have $\{\Phi_j\}_{j=1}^{n_h}$ is a basis for $S_h$. Consequently, $S_h$ is a finite-dimensional subspace of $H_0^1(\mathcal{D})$.\par

\begin{definition}
A triangulation, $\mathcal{T}_h$, is weakly acute if
\[
\int_\mathcal{C}\nabla\Phi_i\cdot\nabla\Phi_j\leq 0\quad\text{for every}\:\: \mathcal{C}\in\mathcal{T}_h\:\:\text{and all}\: 1\leq i,j\leq n_h\:\text{such that}\: i\neq j.
\]
\end{definition}
Note that for $d=1$, this is always true since the hat functions $\Phi_i$
 are piecewise linear on intervals, and their gradients on each interval are constants of opposite sign for neighbouring basis functions. Hence, every triangulation (i.e., every mesh made of intervals) is always weakly acute.
In $d=2$, weak acuteness is equivalent to requiring all interior angles of 
every triangle to be at most $\pi/2$, which ensures that the stiffness matrix 
has non-positive off-diagonal entries (see \cite{Chatzipantelidis_Horvath_Thomee}).  For $d=3$, however, the condition 
becomes substantially more restrictive: in every tetrahedron, the two faces meeting at each edge must form an angle of at most ninety degrees. 
This excludes tetrahedra in which the faces meet at very small or very large angles, preventing highly distorted or nearly flat elements, which would otherwise lead to positive off-diagonal entries 
and destroy monotonicity properties of the discrete Laplacian.

\begin{definition}
    A sequence of triangulations, $\left(\mathcal{T}_{h_i}\right)_{i=1}^\infty$, is called regular (or shape-regular) if there exists a $\rho>0$ such that for every $i\geq1$ and $\mathcal{C}\in\mathcal{T}_{h_i}$, we have
\[
    \text{diam}(B_\mathcal{C})\geq\rho\:\text{diam}(\mathcal{C}),
\]
where $B_\mathcal{C}$ is the largest ball contained within $\mathcal{C}$.
\end{definition}
The  regularity condition prevents elements from becoming arbitrarily thin or degenerate and
ensures a uniform bound on the aspect ratio of the simplices.

For $d=1$ the condition is automatically satisfied, since every element is an
interval and the largest inscribed ball coincides with the interval itself.  In
this case, regularity reduces to requiring that $|\mathcal{C}|\ge \rho' h$ for some
$\rho'>0$, which simply excludes elements of vanishing length.

In $d=2$, the inradius--diameter condition is equivalent to a uniform lower
bound on the area of each triangle:
\[
|\mathcal{C}| \;\ge\; \rho' h^2,
\]
where $h=\max_{\mathcal{C}\in\mathcal{T}_h}\operatorname{diam}(\mathcal{C})$.  Thus regularity
excludes triangles with very small angles or highly anisotropic shapes.

In $d=3$, the volume condition $|\mathcal{C}|\ge \rho' h^3$ is no longer sufficient.  A
tetrahedron can have volume comparable to $h^3$ while still being highly
degenerate (for example, with three nearly collinear vertices or one extremely
small dihedral angle). The inradius--diameter condition
$\operatorname{diam}(B_\mathcal{C})\ge \rho\,\operatorname{diam}(\mathcal{C})$ is needed to rule out
such sliver tetrahedra and ensure that all dihedral angles are bounded away from
$0$ and $\pi$.
We now state our assumptions on the triangulations we will be considering.
\begin{assumption}\label{assumption_2}
    Given a domain $\mathcal{D}\subset\mathbb{R}^d$, which satisfies Assumption \ref{assumption_1} (a), we consider an admissible triangulation $\mathcal{T}_h$ of $\mathcal{D}$ which is weakly acute and there exists a $C_1>0$ such that for each $\mathcal{C}\in\mathcal{T}_h$
    \begin{equation}\label{quasi_condition}
        C_1 \text{diam}(\mathcal{C})\leq \text{diam}(B_\mathcal{C}),
    \end{equation}
    where $B_\mathcal{C}\subset \mathcal{C}$ is the largest ball contained in $\mathcal{C}$.
\end{assumption}

For a single triangulation, condition (\ref{quasi_condition}) is always true since there are finitely many triangles. But, the results that follow will have constants depending on $C_1>0$. This means that for a sequence of shape-regular triangulations, the results will hold with the same constant.

\subsection{Mass and Stiffness Matrices}

For a given admissible triangulation $\mathcal{T}_h$ of $\mathcal{D}\subset\mathbb{R}^d$,  the associated mass matrix $\mathcal{M}$, and stiffness matrix, $\mathcal{S}$, are given by
\[
    \mathcal{M}:=\left[\left(\Phi_j,\Phi_i\right)_{L^2(\mathcal{D})}\right]_{1\leq i,j\leq n_h}\quad\quad\mathcal{S}:=\left[\left(\nabla\Phi_j\cdot\nabla\Phi_i,1\right)_{L^2(\mathcal{D})}\right]_{1\leq i,j\leq n_h}.
\]
It is known that $\mathcal{M}$ is symmetric positive-definite (\cite[Lemma 1]{Fried_intro}) and since we are working with Dirichlet boundary conditions, $\mathcal{S}$ is also symmetric positive-definite. Hence, they each have symmetric positive-definite inverse. Moreover, they both have unique symmetric positive-definite square roots (and so do their inverses).\par
We also introduce the mass-lumped mass matrix
\[
    \mathcal{M}_L:=\left[\delta_{ij}\int_\mathcal{D}\Phi_i\right]_{1\leq i,j\leq n_h},
\]
which is diagonal.

We state a lemma detailing the correspondence between the $L^2(\mathcal{D})$ norm and the Euclidean norm on the coefficients of an element of $S_h$, which will be used later for calculating the $L^2(\mathcal{D})$ norm in numerical experiments.

\begin{lemma}\label{norm_equivalence}
    Suppose we have an admissible triangulation $\mathcal{T}_h$ of domain $\mathcal{D}\subset\mathbb{R}^d$ satisfying Assumption \ref{assumption_1} (a). Then, for every $f,g\in S_h$
    \[
    (f,g)_{L^2(\mathcal{D})}=(\mathcal{M}\tilde{f}, \tilde{g})_2
    \]
    where $(\cdot,\cdot)_2$ is the standard Euclidean inner product and $\tilde{f}$ and $\tilde{g}$ are the vectors of coefficients, $(f_j)_{j=1}^{n_h}$ and $(g_j)_{j=1}^{n_h}$, in the basis expansions
    \[
    f=\sum_{j=1}^{n_h}f_j\Phi_j\quad\text{and}\quad g=\sum_{j=1}^{n_h}g_j\Phi_j.
    \]
    In particular, we have for $f \in S_h$
    \[
    ||f||_{L^2(\mathcal{D})}=||\mathcal{M}^\frac{1}{2}\tilde{f}||_2=\sqrt{(\mathcal{M}\tilde{f},\tilde{f})_2}\,.
    \]
\end{lemma}
\begin{proof}
    Let $f,g\in S_h$ and
    $ 
    f=\sum_{j=1}^{n_h}f_j\Phi_j\quad\text{and}\quad g=\sum_{j=1}^{n_h}g_j\Phi_j.
    $ 
    Then
    \begin{align*}
        (f,g)_{L^2(\mathcal{D})}
        =\sum_{i,j=1}^{n_h}f_j g_i(\Phi_j, \Phi_i)_{L^2(\mathcal{D})}
        =(\mathcal{M}\tilde{f},\tilde{g})_2.
    \end{align*}
    Since $\mathcal{M}$ is symmetric positive definite  \cite[Lemma 1]{Fried_intro}, it has a unique symmetric positive-definite square root $\mathcal{M}^\frac{1}{2}$. Hence by the symmetric property of $\mathcal{M}^\frac{1}{2}$
    \[
    ||f||_{L^2(\mathcal{D})}^2=(\mathcal{M}\tilde{f},\tilde{f})_2
    =(\mathcal{M}^\frac{1}{2}\tilde{f},\mathcal{M}^\frac{1}{2}\tilde{f})_2=||\mathcal{M}^\frac{1}{2}\tilde{f}||_2^2.
    \]
\end{proof}

\subsection{Discretization}

We now follow the exposition in \cite{Djurdjevac} to define the mass lumped discretization of equation (\ref{main_equation}).\par
We first define the nodal interpolant operator $I_h:C_0(\overline{\mathcal{D}})\to S_h$ as
\[
    I_h f(x)=\sum_{j=1}^{n_h}f(P_j)\Phi_j(x), \quad x \in \mathcal{D}.
\]
According to 
\cite{brenner_scott}, for the polyhedral case,
and \cite{Oganesjan},
for the $C^2$ boundary case, we have the following approximation property of the interpolant. If $k=0,1$ and $v\in H^2(\mathcal{D})\cap C_0(\overline{\mathcal{D}})$, then there exists a constant $C_2=C_2(C_1,d,\partial\mathcal{D})>0$, such that
\begin{equation}\label{interpolation_ye}
    ||I_h v - v||_{H^k(\mathcal{D})}\leq C_2h^{2-k}||v||_{H^2(\mathcal{D})}. 
\end{equation}
We now define a semi-inner product, $(\cdot,\cdot)_h$, and its associated semi-norm, $||\cdot||_h$, on $C_0(\overline{\mathcal{D}})$ by
\[
    (w,v)_h:=(I_h(wv),1)_{L^2(\mathcal{D})}\quad\text{for}\quad w,v\in C_0(\overline{\mathcal{D}}),
\]
which is an inner product on $S_h$.\par
We then define the mass lumped discretized Laplacian $\Delta_h:S_h\to S_h$ as
\begin{equation}\label{discrete_laplacian}
(-\Delta_h w,v)_h:=(\nabla w\cdot\nabla v,1)_{L^2(\mathcal{D})}\quad\text{for every}\quad w,v\in S_h.
\end{equation}
Since $-\Delta_h$ is linear, self-adjoint and positive-definite on the finite dimensional space $(S_h,(\cdot,\cdot)_h)$, $\Delta_h$
generates a semigroup of contractions $(E_h(t))_{t\geq0}$ on $(S_h,(\cdot,\cdot)_h)$. In fact we have the following properties of the semigroup.

\begin{lemma}\label{semigroup_lemma}
For each $h>0$, the contractive semigroup $(E_h(t))_{t\geq0}$ on $(S_h,(\cdot,\cdot)_h)$, generated by $\Delta_h$, satisfies the following properties:
\begin{enumerate}[(i)]
\item For every $r\in\mathbb{R}$, $(-\Delta_h)^r:S_h\to S_h$ is well-defined. Moreover, for $\mu\geq0$
\begin{equation}\label{semigroup_lemma_1}
(-\Delta_h)^\mu E_h(t)=E_h(t)(-\Delta_h)^\mu,
\end{equation}
and there exists a constant $C_3=C_3(\mu)>0$, depending only on $\mu$, such that
\begin{equation}\label{semigroup_lemma_2}
    ||(-\Delta_h)^\mu E_h(t)||_h\leq C_3 t^{-\mu}\quad\text{for}\:\: t>0.
\end{equation}
\item For every $0\leq\nu\leq1$, there exists a constant $C_4=C_4(\nu)>0$, depending only on $\nu$, such that
\begin{equation}\label{semigroup_lemma_3}
    \left|\left|(-\Delta_h)^{-\nu}\left(E_h(t)-\operatorname{Id}_{S_h}\right)\right|\right|_h\leq C_4 t^\nu\quad\text{for}\:\: t\geq 0,
\end{equation}
where $\operatorname{Id}_{S_h}$ is the identity on $S_h$.
\item For every $0\leq\rho\leq 1$, there exists a $C_5=C_5(\rho)>0$, depending only on $\rho$, such that
\begin{equation}\label{semigroup_lemma_4}
    \int_{t_1}^{t_2}\left|\left|(-\Delta_h)^{\frac{\rho}{2}}E_h(t_2-s)v\right|\right|_h^2\text{d}s\leq C_5(t_2-t_1)^{1-\rho}||v||_h^2\quad\text{for every}\:\: v\in S_h\:\:\text{and}\:\: 0\leq t_1<t_2.
\end{equation}
\end{enumerate}
\end{lemma}
\begin{proof}
    Since $-\Delta_h:S_h\to S_h$ is linear, self-adjoint and positive-definite on the finite dimensional Hilbert space $(S_h,(\cdot,\cdot)_h)$, the conditions in \cite[Appendix B.2]{Kruse} are met. Namely, there exists an orthonormal eigenbasis, $(\phi_{h,i})_{i=1}^{n_h}$, of the operator $-\Delta_h$ on $(S_h, (\cdot,\cdot)_h)$. The corresponding eigenvalues, $(\lambda_{h,i})_{i=1}^{n_h}$, are strictly positive and the operators $(-\Delta_h)^r:S_h\to S_h$, for $r\in\mathbb{R}$, are defined by
    \[
    (-\Delta_h)^r x:=\sum_{i=1}^{n_h}\lambda_{h,i}^r\:(x,\phi_{h,i})_h\:\phi_{h,i}\quad\text{for}\quad x\in S_h.
    \]
   Lemma \ref{semigroup_lemma} then follows from the more general result \cite[Lemma B.9]{Kruse}.
    In particular, it is important to note the constants in (\ref{semigroup_lemma_2}), (\ref{semigroup_lemma_3}) and (\ref{semigroup_lemma_4}) do not depend on $h$ as the constants in \cite[Lemma B.9]{Kruse} are explicitly
    \begin{align*}
        C_3(\mu)&:=\sup_{x\in[0,\infty)}x^\mu e^{-x},\\
        C_4(\nu)&:=\sup_{x\in[0,\infty)}x^{-\nu}(1-e^{-x}),\\
        C_5(\rho)&:=2^{-\rho}\sup_{x\in[0,\infty)}x^{\rho-1}(1-e^{-x}).
    \end{align*}
    \end{proof}
Another important property of the $||\cdot||_h$ norm is the existence of a constant $C>0$, independent of $h$, such that 
\begin{equation}\label{h_norm_equivalence}
    ||v||_{L^2(\mathcal{D})}\leq||v||_h\leq C||v||_{L^2(\mathcal{D})}\quad\text{for every}\quad v\in S_h.
\end{equation}
For the proof see \cite[Lemma 2]{Djurdjevac}.

The mass-lumped finite element discretization  of equation (\ref{main_equation}) is given by:
\begin{equation}\label{main_disc_eqn}
\begin{split}
\text{d}u_h(t)&=\Delta_h u_h(t)\text{d}t+\sum_{k=1}^KI_h(f(u_h(t))e_k)\text{d}B^k_t\quad\text{for}\quad t\in(0,T]\\
u_h(0)&=I_h u_0\,.
\end{split}
\end{equation}

    In order to treat the nonlinear term $f(u)$, we introduce a new function $g:\mathbb{R}\to\mathbb{R}$ by
    \[
        g(s)=\frac{f(s)}{s}\mathbbm{1}_{\{s\neq0\}}+f'(0)\mathbbm{1}_{\{s=0\}}=\int_0^1 f'(rs)\text{d}r\,.
    \]
    The function $g$ is continuous, since $f$ is $C^1$, and satisfies $f(s)=g(s)s$ for every $s\in\mathbb{R}$. Moreover, $g$ is bounded, since $f$ is globally Lipschitz, and
\[
||g||_{L^\infty(\mathbb{R})}\leq L_f,
\]
since $f(0)=0$.\par
Similarly as in  \cite[Section 3.1]{Djurdjevac}, equation (\ref{main_disc_eqn})  has an equivalent form as a system of SDEs. Namely, since $u_h \in S_h$, we can write  
\[
u_h(t)=\sum_{i=1}^{n_h}U_h(t)_i\Phi_i\quad\text{for}\quad t\in[0,T].
\]
Then coefficients $U_h(t)_i$, for $1\leq i\leq n_h$, are solving
\begin{equation}\label{main_disc_vec_eqn}
\begin{split}
    \text{d}U_h(t)&=-\mathcal{M}_L^{-1}\mathcal{S}U_h(t)\text{d}t+\sum_{k=1}^K D_{g(U_h(t))}D_{e_k}U_h(t)\text{d}B^k_t\quad\text{for}\quad t\in(0,T]\\
U_h(0)&=U_{h,0},
\end{split}
\end{equation}
where $U_{h,0}=(u_0(P_j))_{j=1}^{n_h}$, $D_{g(U_h(t))}$ is the diagonal matrix $\left[g(U_h(t)_i)\delta_{ij}\right]_{1\leq i,j\leq n_h}$ and $D_{e_k}$ is the diagonal matrix $\left[e_k(P_i)\delta_{ij}\right]_{1\leq i,j\leq n_h}$.
By the mild solution theory for SPDE's, \cite[Theorem 2.25]{Kruse}, equations (\ref{main_disc_vec_eqn}) and (\ref{main_disc_eqn}) have unique mild solutions.
These solutions are  given by
\begin{equation}\label{mild_sh}
    u_h(t)=E_h(t)I_h u_0 + \sum_{k=1}^K\int_0^t E_h(t-s)I_h(f(u_h(s)) e_k)\text{d}B^k_s\quad\text{for}\quad t\in[0,T],
\end{equation}
and
\[
    U_h(t)=\exp\left(-t\mathcal{M}_L^{-1}\mathcal{S}\right)U_{h,0}+\sum_{k=1}^K\int_0^t\exp\left(-(t-s)\mathcal{M}_L^{-1}\mathcal{S}\right)D_{g(U_h(s))} D_{e_k}U_h(s)\text{d}B^k_s\quad\text{for}\quad t\in[0,T].
\]

The following properties of the $||\cdot||_h$ norm will prove useful in later proofs.
\begin{lemma}\label{properties_of_h_norm}
    Let $h>0$. Then for every $v\in C_0(\overline{\mathcal{D}})$, we have
    \begin{equation}\label{h_norm_iso}
        ||I_h(v)||_h=||v||_h\leq||v||_{L^\infty(\mathcal{D})}\sqrt{|\mathcal{D}|}. 
    \end{equation}
    In particular, for $s_h\in S_h$ and $v\in C(\overline{\mathcal{D}})$
    \begin{equation}\label{best_inequality}
        ||I_h(f(s_h)v)||_h\leq L_f||v||_{L^\infty(\mathcal{D})}||s_h||_h.
    \end{equation}
    Moreover, for $0\leq s<t\leq r\leq T$
    \begin{equation}\label{ito_app1}
        \mathbb{E}\left[\left|\left|\sum_{k=1}^K\int_s^t E_h(r-\sigma)I_h(f(u_h(\sigma)) e_k)\text{d}B^k_\sigma\right|\right|_{h}^2\right]\leq L_f^2 K_e^2\int_s^t\mathbb{E}\left[||u_h(\sigma)||_h^2\right]\text{d}\sigma.
    \end{equation}    
\end{lemma}
\begin{proof}
    The equality in (\ref{h_norm_iso}) follows from the fact that for $1\leq i \leq n_h$, we have $I_h(v)(P_i) = v(P_i)$. In fact
    \[
||I_h(v)||_h^2=\int_\mathcal{D}I_h(I_h(v)^2)(x)dx=\int_\mathcal{D}\sum_{i=1}^{n_h}(I_h(v)(P_i))^2\Phi_i(x) dx=\int_\mathcal{D}\sum_{i=1}^{n_h}v(P_i)^2\Phi_i(x)dx=\int_\mathcal{D}I_h(v^2)(x)dx=||v||_h^2.
    \]
    The inequality in (\ref{h_norm_iso}) follows from the fact that if $\phi\in C_0(\overline{\mathcal{D}})$ such that $\phi\leq C$ for some $C>0$, then $I_h(\phi)(x)\leq C$ for all $x \in \mathcal{D}$. Hence
\[ 
||v||_h^2=\int_\mathcal{D}I_h(v^2)dx\leq\int_\mathcal{D}||v||_{L^\infty}^2dx=||v||_{L^\infty}^2|\mathcal{D}|.
    \]
    Inequality (\ref{best_inequality}) follows from the fact that if $\phi,\psi\in C_0(\overline{\mathcal{D}})$ with $\phi\leq\psi$, then $I_h(\phi)\leq I_h(\psi)$.
    Hence, using (\ref{h_norm_iso}) and the conditions on $f$, we get
    \[
    ||I_h(f(s_h)v)||_h^2=||f(s_h)v||_h^2=\int_{\mathcal{D}}I_h((f(s_h)v)^2)dx\leq L_f^2||v||_{L^\infty}^2\int_\mathcal{D}I_h(s_h^2)dx=L_f^2||v||_{L^\infty}^2||s_h||_h^2.
    \]
    To prove (\ref{ito_app1}), we use a Burkholder-Davis-Gundy-type inequality for Hilbert space valued stochastic integrals (see  \cite[Proposition 2.12]{Kruse}). We first consider the map $\Psi:[0,r] \times \Omega \to L_2(\mathcal{H},S_h)$ defined by $(\sigma,\omega)\mapsto E_h(r-\sigma)I_h(f(u_h(\sigma,\omega)) \cdot)$. Here, $\mathcal{H}=\text{Span}\{e_1,...,e_K\}$ endowed with the $L^2(\mathcal{D})$ norm, $S_h$ is endowed with the $||\cdot||_h$ norm and $L_2(\mathcal{H},S_h)$ is the space of Hilbert-Schmidt operators from $\mathcal{H}$ to $S_h$ (see \cite[Definition 2.8]{Kruse}). Using the definition of the Hilbert-Schmidt norm, the contractive properties of the semi-group and inequality (\ref{best_inequality}), we get
    \begin{align*}
        ||\Psi(\sigma,\omega)||_{L_2(\mathcal{H},S_h)}^2:&=\sum_{k=1}^K||\Psi(\sigma,\omega)(e_k)||_h^2=\sum_{k=1}^K||E_h(r-\sigma)I_h(f(u_h(\sigma,\omega))e_k)||_h^2\\
        &\leq\sum_{k=1}^K||I_h(f(u_h(\sigma,\omega))e_k)||_h^2\leq L_f^2\left(\sum_{k=1}^K||e_k||_{L^\infty}^2\right)||u_h(\sigma,\omega)||_h^2\\
        &=L_f^2K_e^2||u_h(\sigma,\omega)||_h^2.
    \end{align*}
    We use this inequality, a representation of the Hilbert space valued stochastic integral (\cite[Proposition 2.4.5]{Liu_Rockner}) and the  Burkholder-Davis-Gundy-type inequality to conclude 
    \begin{align*}
\mathbb{E}\left[\left|\left|\sum_{k=1}^K\int_s^t E_h(r-\sigma)I_h(f(u_h(\sigma)) e_k)\text{d}B^k_\sigma\right|\right|_{h}^2\right]&=\mathbb{E}\left[\left|\left|\sum_{k=1}^K\int_s^t \Psi(\sigma)( e_k)\text{d}B^k_\sigma\right|\right|_{h}^2\right]\\
    &\leq\mathbb{E}\left[\int_s^t||\Psi(\sigma)||_{L_2(\mathcal{H},S_h)}^2\text{d}\sigma\right]\leq L_f^2 K_e^2\int_s^t\mathbb{E}\left[||u_h(\sigma)||_h^2\right]\text{d}\sigma.
    \end{align*}
\end{proof}
We now show moment bounds and temporal regularity of the solutions $u_h$. Crucially, these are independent of $h$. The proof of the temporal regularity follows the proof of \cite[Theorem 2.31]{Kruse}, only here we focus in particular on the dependence on $h$.

\begin{theorem}\label{uh_properties}
    Let Assumption \ref{assumption_1} and Assumption \ref{assumption_2} hold. Then
    \begin{equation}\label{uh_momentbounds}
        \sup_{t\in[0,T]}\mathbb{E}\left[||u_h(t)||_{L^2(\mathcal{D})}^2\right]\leq C
    \end{equation}
    where $C = C\left(||u_0||_{L^\infty},|\mathcal{D}|,L_f,K_e,T\right)$ is independent of $h$. Moreover, for every $s,t\in[0,T]$
    \begin{equation}\label{uh_temporal}
        \mathbb{E}\left[||u_h(t)-u_h(s)||_{L^2(\mathcal{D})}^2\right]^\frac{1}{2}\leq C|t-s|^\frac{1}{2},
    \end{equation}
where $C=C\left(C_1, d, \text{diam}(\mathcal{D}), \partial\mathcal{D},|\mathcal{D}|,||u_0||_{L^\infty},||u_0||_{H^2},L_f,K_e,T\right)$ is independent of $h$.
\end{theorem}
\begin{proof}
    Starting from the mild solution, (\ref{mild_sh}), we may use the contractive properties of the semigroup and inequalities (\ref{ito_app1}) and (\ref{h_norm_iso}) to get
    \begin{align*}
        \mathbb{E}\left[||u_h(t)||_{h}^2\right]&\leq 2\left(||E_h(t)I_h u_0||_{h}^2+\mathbb{E}\left[\left|\left|\sum_{k=1}^K\int_0^t E_h(t-s)I_h(f(u_h(s)) e_k)\text{d}B^k_s\right|\right|_{h}^2\right]\right)\\
        &\leq 2\left(||I_h u_0||_{h}^2+L_f^2 K_e^2\int_0^t\mathbb{E}\left[||u_h(s)||_h^2\right]\text{d}s\right)\\
        &\leq 2||u_0||_{L^\infty}^2|\mathcal{D}|+2L_f^2 K_e^2\int_0^t\mathbb{E}\left[||u_h(s)||_h^2\right]\text{d}s.
    \end{align*}
    Then by Gr\"onwall's inequality
    \begin{align*}
        \mathbb{E}\left[||u_h(t)||_{h}^2\right]&\leq 2||u_0||_{L^\infty}^2|\mathcal{D}|\left(1+2L_f^2 K_e^2 t\exp(2L_f^2 K_e^2 t)\right).
    \end{align*}
    Using the equivalence of the $h$-norms, (\ref{h_norm_equivalence}), and then taking the supremum over time, yields
    \begin{equation}\label{h_sup_norm}
    \sup_{t\in[0,T]}\mathbb{E}\left[||u_h(t)||_{L^2}^2\right]\leq \sup_{t\in[0,T]}\mathbb{E}\left[||u_h(t)||_{h}^2\right]\leq C,
    \end{equation}
    where
    \begin{equation}\label{h_sup_norm_constant}
        C=C\left(||u_0||_{L^\infty},|\mathcal{D}|,L_f,K_e,T\right)=2||u_0||_{L^\infty}^2|\mathcal{D}|\left(1+2L_f^2 K_e^2 T\exp(2L_f^2 K_e^2 T)\right).
    \end{equation}

    Now, let $0\leq s<t\leq T$. Again, using the mild formulation, (\ref{mild_sh}), and the triangle inequality, we get
    \begin{equation}\label{terms1}
    \begin{split}
        \mathbb{E}\left[\left|\left|u_h(t)-u_h(s)\right|\right|_{h}^2\right]^\frac{1}{2}&\leq ||(E_h(t)-E_h(s))I_h u_0||_{h}+\mathbb{E}\left[\left|\left|\sum_{k=1}^K\int_s^t E_h(t-\sigma)I_h(f(u_h(\sigma)) e_k)\text{d}B^k_\sigma\right|\right|_{h}^2\right]^\frac{1}{2}\\
     &\quad\quad\quad\quad\quad\quad+\mathbb{E}\left[\left|\left|\sum_{k=1}^K\int_0^s (E_h(t-\sigma)-E_h(s-\sigma))I_h(f(u_h(\sigma)) e_k)\text{d}B^k_\sigma\right|\right|_{h}^2\right]^\frac{1}{2}\\
     &=:A_1+A_2+A_3.
    \end{split}
    \end{equation}
    To estimate $A_1$, we use Lemma \ref{semigroup_lemma}, particularly (\ref{semigroup_lemma_1}) and (\ref{semigroup_lemma_3}), the contractive properties of the semigroup, the definition of the discrete Laplacian, (\ref{discrete_laplacian}), and the approximation property of the interpolant, (\ref{interpolation_ye}), to get
    \begin{equation}\label{A_1}
    \begin{split}
        A_1&=||(-\Delta_h)^{-\frac{1}{2}}(E_h(t-s)-I_{S_h})(-\Delta_h)^\frac{1}{2}E_h(s)I_h u_0||_{h}\\
        &\leq C_4|t-s|^\frac{1}{2}||(-\Delta_h)^\frac{1}{2}E_h(s)I_h u_0||_{h}\\
        &\leq C_4|t-s|^\frac{1}{2}||(-\Delta_h)^\frac{1}{2}I_h u_0||_{h}=C_4|t-s|^\frac{1}{2}||\nabla I_h u_0||_{L^2}\\
        &\leq C_4|t-s|^\frac{1}{2}\left(||\nabla (I_h u_0 - u_0)||_{L^2}+||\nabla u_0||_{L^2}\right)\\
        &\leq C_4|t-s|^\frac{1}{2}\left(C_2h||u_0||_{H^2}+||u_0||_{H^1}\right)\\
        &\leq C |t-s|^\frac{1}{2}.
    \end{split}
    \end{equation}
    where 
    \begin{equation}\label{constant1}
        C=C\left(C_1,d,\text{diam}(\mathcal{D}),\partial\mathcal{D},||u_0||_{H^2}\right)=C_4\left(\frac{1}{2}\right) \left(C_2(C_1,d,\partial\mathcal{D})\text{diam}(\mathcal{D}) +1\right)||u_0||_{H^2}.
    \end{equation}

To estimate $A_2$, we use inequality (\ref{ito_app1}) and the uniform $h$-norm bound, (\ref{h_sup_norm}), to get
\begin{equation}\label{A_2}
    \begin{split}
        A_2^2&\leq L_f^2 K_e^2\int_s^t\mathbb{E}\left[||u_h(\sigma)||_h^2\right]\text{d}\sigma\leq C|t-s|,
    \end{split}
\end{equation}
where
\begin{equation}\label{constant2}
    C=C\left(||u_0||_{L^\infty},|\mathcal{D}|,L_f,K_e,T\right)=2 L_f^2 K_e^2 ||u_0||_{L^\infty}^2|\mathcal{D}|\left(1+2L_f^2 K_e^2 T\exp(2L_f^2 K_e^2 T)\right).
\end{equation}
To estimate $A_3$, we first use the Burkholder-Davis-Gundy-type inequality for Hilbert space valued stochastic integrals (\cite[Proposition 2.12]{Kruse}) applied to $(\sigma,\omega)\mapsto (E_h(t-\sigma)-E_h(s-\sigma))I_h(f(u_h(\sigma,\omega)) \cdot)$, then Lemma \ref{semigroup_lemma}, particularly (\ref{semigroup_lemma_1}), (\ref{semigroup_lemma_2}) and (\ref{semigroup_lemma_3}), and inequalities (\ref{best_inequality}) and (\ref{h_sup_norm}) to get 
\begin{equation}\label{A_3_quarter}
    \begin{split}
        A_3^2&\leq \sum_{k=1}^K\int_0^s \mathbb{E}\left[\left|\left|(E_h(t-\sigma)-E_h(s-\sigma))I_h(f(u_h(\sigma)) e_k)\right|\right|_h^2\right]\text{d}\sigma\\
        &=\sum_{k=1}^K\int_0^s \mathbb{E}\left[\left|\left|(-\Delta_h)^{-\frac{1}{4}}(E_h(t-s)-\operatorname{Id}_{S_h})(-\Delta_h)^\frac{1}{4}E_h(s-\sigma)I_h(f(u_h(\sigma)) e_k)\right|\right|_h^2\right]\text{d}\sigma\\
        &\leq C_4^2|t-s|^\frac{1}{2}\sum_{k=1}^K\int_0^s\mathbb{E}\left[\left|\left|(-\Delta_h)^\frac{1}{4}E_h(s-\sigma)I_h(f(u_h(\sigma)) e_k)\right|\right|_h^2\right]\text{d}\sigma\\
        &\leq C_4^2 C_3^2|t-s|^\frac{1}{2}\sum_{k=1}^K\int_0^s\frac{1}{\sqrt{s-\sigma}}\mathbb{E}\left[\left|\left|I_h(f(u_h(\sigma)) e_k)\right|\right|_h^2\right]\text{d}\sigma\\
        &\leq C_4^2 C_3^2 L_f^2K_e^2|t-s|^\frac{1}{2}\int_0^s\frac{1}{\sqrt{s-\sigma}}\mathbb{E}\left[\left|\left|u_h(\sigma)\right|\right|_{h}^2\right]\text{d}\sigma\\
        &\leq C|t-s|^\frac{1}{2},
    \end{split}
\end{equation}
where
\begin{equation}\label{constant3}
    C=C\left(||u_0||_{L^\infty},|\mathcal{D}|,L_f,K_e,T\right)=4C_4\left(\frac{1}{4}\right)^2 C_3\left(\frac{1}{4}\right)^2 L_f^2K_e^2\sqrt{T}||u_0||_{L^\infty}^2|\mathcal{D}|\left(1+2L_f^2 K_e^2 T\exp(2L_f^2 K_e^2 T)\right)\,.
\end{equation}
Combining (\ref{A_1}), (\ref{A_2}) and (\ref{A_3_quarter}) yields 
\begin{equation}\label{holder_14}
     \mathbb{E}\left[||u_h(t)-u_h(s)||_{h}^2\right]^\frac{1}{2}\leq C|t-s|^\frac{1}{4}\:\:\text{for every}\:\: s,t\in[0,T],
\end{equation}
where the constant $C=C\left(C_1, d, \text{diam}(\mathcal{D}), \partial\mathcal{D},|\mathcal{D}|,||u_0||_{L^\infty},||u_0||_{H^2},L_f,K_e,T\right)$ can be inferred from (\ref{terms1}), (\ref{constant1}), (\ref{constant2}) and (\ref{constant3}).\par
We now revisit estimating $A_3$, this time using inequality (\ref{semigroup_lemma_3}) with exponent $\frac{1}{2}$ to yield
\begin{equation}\label{A_3}
\begin{split}
    A_3^2&\leq \sum_{k=1}^K\int_0^s \mathbb{E}\left[\left|\left|(-\Delta_h)^{-\frac{1}{2}}(E_h(t-s)-\operatorname{Id}_{S_h})(-\Delta_h)^\frac{1}{2}E_h(s-\sigma)I_h(f(u_h(\sigma)) e_k)\right|\right|_h^2\right]\text{d}\sigma\\
    &\leq 2C_4^2|t-s|\sum_{k=1}^K\Bigg(\int_0^s\mathbb{E}\left[\left|\left|(-\Delta_h)^\frac{1}{2}E_h(s-\sigma)I_h((f(u_h(\sigma))-f(u_h(s))) e_k)\right|\right|_h^2\right]\text{d}\sigma\\
    &\quad\quad\quad\quad\quad\quad\quad\quad\quad\quad+\int_0^s\mathbb{E}\left[\left|\left|(-\Delta_h)^\frac{1}{2}E_h(s-\sigma)I_h(f(u_h(s)) e_k)\right|\right|_h^2\right]\text{d}\sigma\Bigg)\\
    &=:2C_4^2|t-s|\sum_{k=1}^K\left(A_{3,1,k}+A_{3,2,k}\right).
    \end{split}
    \end{equation}
To estimate $A_{3,1,k}$, for $1\leq k\leq K$, we use inequality (\ref{semigroup_lemma_2}), the Lipschitz continuity of $f$ and the $\frac{1}{4}$-H\"{o}lder continuity, (\ref{holder_14}), to get
\begin{equation}\label{A_3_1_k}
    \begin{split}
        A_{3,1,k}&\leq C_3^2\int_0^s\frac{1}{s-\sigma}\mathbb{E}\left[\left|\left|I_h((f(u_h(\sigma))-f(u_h(s))) e_k)\right|\right|_{h}^2\right]\text{d}\sigma\\
        &\leq C_3^2 L_f^2||e_k||_{L^\infty}^2 \int_0^s\frac{1}{s-\sigma}\mathbb{E}[||u_h(\sigma)-u_h(s)||_{h}^2]\text{d}\sigma\\
        &\leq C_3^2 L_f^2||e_k||_{L^\infty}^2 C\int_0^s\frac{1}{\sqrt{s-\sigma}}\text{d}\sigma\\
        &\leq 2C_3^2 L_f^2||e_k||_{L^\infty}^2 \sqrt{T}C,
    \end{split}
\end{equation}
where $C=C\left(C_1, d, \text{diam}(\mathcal{D}), \partial\mathcal{D},|\mathcal{D}|,||u_0||_{L^\infty},||u_0||_{H^2},L_f,K_e,T\right)$ is from (\ref{holder_14}). 
To estimate $A_{3,2,k}$, for $1\leq k\leq K$, we use inequalities (\ref{semigroup_lemma_4}), (\ref{h_sup_norm}) and (\ref{best_inequality}) to get
\begin{equation}\label{A_3_2_k}
    \begin{split}
        A_{3,2,k}&=\mathbb{E}\left[\int_0^s\left|\left|(-\Delta_h)^\frac{1}{2}E_h(s-\sigma)I_h(f(u_h(s)) e_k)\right|\right|_h^2\text{d}\sigma\right]\\
        &\leq C_5\mathbb{E}\left[||I_h(f(u_h(s))e_k)||_h^2\right]\\
        &\leq C_5 L_f^2||e_k||_{L^\infty}^2C,
    \end{split}
\end{equation}
where $C=C\left(||u_0||_{L^\infty},|\mathcal{D}|,L_f,K_e,T\right)$ is from (\ref{h_sup_norm}).
Combining (\ref{A_3}), (\ref{A_3_1_k}) and (\ref{A_3_2_k}) then yields
    \begin{equation}\label{A_3_final}
        \begin{split}
A_3^2\leq C|t-s|,
\end{split}
\end{equation}
where the constant  $C=C\left(C_1, d, \text{diam}(\mathcal{D}), \partial\mathcal{D},|\mathcal{D}|,||u_0||_{L^\infty},||u_0||_{H^2},L_f,K_e,T\right)$ can be inferred from (\ref{A_3}), (\ref{A_3_1_k}), (\ref{A_3_2_k}), (\ref{holder_14}) and (\ref{h_sup_norm_constant}).\par
Combining (\ref{A_1}), (\ref{A_2}), (\ref{A_3_final}) and the $h$-norm equivalence, (\ref{h_norm_equivalence}), then yields the result.

\end{proof}

As shown in Theorem 4 of \cite{Djurdjevac}, the solution $u_h$ of (\ref{main_disc_eqn}) is nonnegative under Assumption \ref{assumption_1} and Assumption \ref{assumption_2}. Moreover, under an additional assumption on the regularity of the solution $u$, of (\ref{main_equation}), we have a result on the strong error between $u_h$ and $u$,  a more detailed discussion is given in \cite{Djurdjevac}.
\begin{assumption}\label{assumption_3}
    The solution $u$ of equation (\ref{main_equation}) satisfies $u\in L^2(\Omega;L^2([0,T];H^2(\mathcal{D})\cap H^1_0(\mathcal{D})))$ and $\Delta u\in L^2(\Omega;L^2([0,T];W^{2,\infty}(\mathcal{D})\cap W^{1,\infty}_0(\mathcal{D})))$.
\end{assumption}

\begin{theorem}\label{strong_error_semi}
    Let Assumptions \ref{assumption_1}, \ref{assumption_2} and \ref{assumption_3} hold. Then the solutions $u$ of (\ref{main_equation}) and $u_h$ of (\ref{main_disc_eqn}) satisfy
    \begin{equation}\label{trve_solution_error}
        \mathbb{E}\left[\left|\left|u(t)-u_h(t)\right|\right|_{L^2(\mathcal{D})}^2\right]+2\int_0^t\mathbb{E}\left[\left|\left|\nabla(u-u_h)(s)\right|\right|_{L^2(\mathcal{D})}^2\right]\text{d}s\leq C_6 h^2\quad\text{for}\quad t\in[0,T],
    \end{equation}
    where $C_6=C_6(T,u)$ is independent of the mesh size $h$.
\end{theorem}
\section{The Splitting Scheme and Nonnegativity}\label{section4}

We now present the fully discrete Lie-Trotter splitting scheme. The approach consists in separating the deterministic and stochastic contributions in (\ref{main_disc_vec_eqn}), solving the resulting subproblems individually, and recombining them appropriately. The time discretization method is the same as in \cite{BCU_1}, but here we apply it to a different semi-discrete equation or system of SDEs. Moreover, we show that the resulting scheme is nonnegativity-preserving.

 Let $\tau:=\frac{T}{M}>0$ (for some fixed $M\in\mathbb{N}$) be the time step and let $t_m=m\tau$ for $m=0,...,M$. We define the Lie-Trotter splitting scheme $(U^{LT,h}_m)_{m=0}^M\subset\mathbb{R}^{n_h}$, which produces approximations $U^{LT,h}_m$ of $U_h(t_m)$ for $m=0,...,M$, where $U_h$ is the solution to (\ref{main_disc_vec_eqn}). We define $u^{LT,h}_m$ to be the corresponding function in $S_h$.\par
First we define $U^{LT,h}_0:=U_{h,0}=(u_0(P_j))_{j=1}^{n_h}$.
Given $U^{LT,h}_m$, for some $m=0,...,M-1$, the numerical solution at the next grid time $U^{LT,h}_{m+1}$ is given by successively solving two subsystems, both of which are written in vector form (solutions lie in $\mathbb{R}^{n_h}$).\par
The first subsystem corresponds to a slightly modified version of the stochastic component of \eqref{main_disc_vec_eqn} and is given by 
\begin{equation}\label{v_1}
\begin{split}
    \text{d}V_m^{M,h,1}(t)&=\sum_{k=1}^K D_{g\left(U_m^{LT,h}\right)} D_{e_k}V_m^{M,h,1}(t)\text{d}B^k_t\quad\text{for}\quad t\in(t_m,t_{m+1}],\\
    V_m^{M,h,1}(t_m)&=U^{LT,h}_m,
    \end{split}
\end{equation}
where $D_{g\left(U_m^{LT,h}\right)}=\left[g\left((U_{m}^{LT,h})_i\right)\delta_{ij}\right]_{1\leq i,j\leq n_h}$. The difference with the stochastic part of (\ref{main_disc_vec_eqn}) is the value inside the $g$ is frozen at the most recent time step.\par
The second subsystem is solving the deterministic part of (\ref{main_disc_vec_eqn}) and is given by
\begin{align*}
        \text{d}V_m^{M,h,2}(t)&=-\mathcal{M}_L^{-1}\mathcal{S}V_m^{M,h,2}(t)\text{d}t\quad\text{for}\quad t\in(t_m,t_{m+1}],\\
        V_m^{M,h,2}(t_m)&=V_m^{M,h,1}(t_{m+1}).
\end{align*}

We then define $U^{LT,h}_{m+1}=V_m^{M,h,2}(t_{m+1})$.\par

\vspace{2mm}

The solutions to the two subsystems are in fact explicitly known. An application of the general It\^o formula yields
\begin{equation}\label{v_1_explicit}
    V_m^{M,h,1}(t)=\exp\left(\sum_{k=1}^K(B^k_t-B^k_{t_m})D_{g\left(U_m^{LT,h}\right)} D_{e_k} - \frac{t-t_m}{2}\sum_{k=1}^K D_{g\left(U_m^{LT,h}\right)}^2D_{e_k}^2\right)U^{LT,h}_m\quad\text{for}\quad t\in[t_m,t_{m+1}].
\end{equation}
Using the theory of semigroups, we also have
\[
    V_m^{M,h,2}(t)=\exp\left(-(t-t_m)\mathcal{M}_L^{-1}\mathcal{S}\right)V_m^{M,h,1}(t_{m+1})\quad\text{for}\quad t\in[t_m,t_{m+1}].
\]

Hence, we have
\begin{equation}\label{lie_trotter}
\begin{split}
    U^{LT,h}_0&=U_{h,0},\\
    U^{LT,h}_{m+1}&=\exp\left(-\tau\mathcal{M}_L^{-1}\mathcal{S}\right)\exp\left(\sum_{k=1}^K\delta B^k_m D_{g\left(U_m^{LT,h}\right)} D_{e_k} - \frac{\tau}{2}\sum_{k=1}^K D_{g\left(U_m^{LT,h}\right)}^2 D_{e_k}^2\right)U^{LT,h}_m,
    \end{split}
\end{equation}
where $\delta B^k_m=B^k_{t_{m+1}}-B^k_{t_m}\sim\mathcal{N}(0,\tau)$. Since the matrices inside the second exponential are diagonal, we may also write
\[
    U^{LT,h}_{m+1}=\exp\left(-\tau\mathcal{M}_L^{-1}\mathcal{S}\right)\left(\exp\left(g\left((U_{m}^{LT,h})_j\right)\sum_{k=1}^K\delta B^k_m e_k(P_j) - \frac{\tau}{2}g\left((U_{m}^{LT,h})_j\right)^2\sum_{k=1}^K e_k(P_j)^2\right)(U^{LT,h}_{m})_j\right)_{j=1}^{n_h}.
\]

Next, we are interested in showing $U^{LT,h}_m\geq0$ for every $m=0,...,M$. The only part of the scheme that requires careful attention is the matrix $\exp(-\tau\mathcal{M}_L^{-1}\mathcal{S})$. This is because the second part is just an exponential multiplied by the nonnegative solution at the previous timestep.
\begin{theorem}\label{splitting_positivity}
    Under Assumptions \ref{assumption_1} and \ref{assumption_2}, the fully discrete numerical scheme $(U_m^{LT,h})_{m=0}^M$ in (\ref{lie_trotter}) is nonnegative.
\end{theorem}
\begin{proof}
    Due to the assumption of weakly acute triangulations, we know that $-\tau\mathcal{M}_L^{-1}\mathcal{S}$ is a Metzler matrix, i.e. all the non-diagonal elements are nonnegative. It is known that matrix exponentials of Metzler matrices are nonnegative. This is because a Metzler matrix $A$ can be written as $A = B-cI$ for some nonnegative matrix $B$ and $c\geq0$. Since $B$ and $-cI$ commute, we have $\exp(A)=\exp(B)\exp(-cI)$. The matrix $\exp(B)$ is nonnegative due to the Taylor series definition of the matrix exponential and $\exp(-cI)=\exp(-c)I$ is also nonnegative. Hence so is $\exp(A)$. Hence the fully discrete numerical scheme $(U_m^{LT,h})_{m=0}^M$ in (\ref{lie_trotter}) is nonnegative.
\end{proof}
\section{Convergence of the Splitting Scheme}\label{section5}

We now show bounds on the second moment of the Lie-Trotter splitting scheme and show strong $L^2$-convergence to the solutions of (\ref{main_disc_eqn}) and (\ref{main_equation}). The method presented adapts the arguments made in Section 5 of \cite{BCU_1}. Here we make use of the semi-group estimates in Lemma \ref{semigroup_lemma}, whereas in \cite{BCU_1} they make use of the estimates of a discrete heat kernel. This also explains why \cite{BCU_1} requires a CFL-type condition, while no such condition is needed in our setting. Moreover, we analyse an error which is $L^2$ in probability and space and $L^\infty$ in time, whereas in \cite{BCU_1} they analyse an error which is $L^2$ in probability and discretely $L^\infty$ in time and space. Because of these differences, we are able to prove stronger rates and, in the case of showing moment bounds, give a simpler proof, but at the expense of having simpler noise.\par
\vspace{2mm}
We define the piecewise constant time‐projection function $l^M:[0,T]\to\mathbb{R}$ that maps each time 
$t$ to the most recent time grid point $t_m$: $l^M(t):= t_m$ for $t\in[t_m,t_{m+1})$ when $m=0,...,M-1$ and $l^M(T)=T=t_M$.\par
\vspace{2mm}

\subsection{The Auxiliary Process}

In preparation for the convergence analysis of the Lie-Trotter splitting scheme, we introduce a continuous‐time extension of the discrete scheme. The process is first described at a formal level, after which its well‐posedness is rigorously established in Proposition \ref{aux_process}.

\begin{definition}\label{aux_process}
Let Assumptions \ref{assumption_1} and \ref{assumption_2} hold. Let $M\in\mathbb{N}$ with $\tau:=\frac{T}{M}$ and define the auxiliary process $U^{LT,h} : [0,T]\times\Omega \to \mathbb{R}^{n_h}$ to be
\begin{align*}
    U^{LT,h}(t) &:= \exp\left(-l^M(t)\mathcal{M}_L^{-1}\mathcal{S}\right)U_{h,0}\\
    &\quad\quad\quad\quad+ \sum_{k=1}^K\int_0^t \exp\left(-(l^M(t) - l^M(s))\mathcal{M}_L^{-1}\mathcal{S}\right)D_{g\left(U^{LT,h}\left(l^M(s)\right)\right)} D_{e_k} U^{LT,h}(s) dB^k_s,
\end{align*}
for $t\in[0,T]$ and with  $D_{g\left(U^{LT,h}\left(l^M(s)\right)\right)}=\left[g\left(\left(U^{LT,h}\left(l^M(s)\right)\right)_i\right)\delta_{ij}\right]_{1\leq i,j\leq n_h}$.\par
It can also be written in its equivalent $S_h$-valued form $u^{LT,h} : [0,T]\times\Omega \to S_h$, 
\begin{equation}\label{aux_process_equ}
    u^{LT,h}(t) := E_h(l^M(t))I_h u_0 + \sum_{k=1}^K\int_0^t E_h(l^M(t) - l^M(s)) I_h\left(g\left(u^{LT,h}\left(l^M(s)\right)\right)u^{LT,h}(s)e_k\right) dB^k_s,
\end{equation}
for $t\in[0,T]$. 
\end{definition}

For notational convenience, we drop the $h$ superscript in proofs and write $U^{LT}:=U^{LT,h}$, $u^{LT}:=u^{LT,h}$, $U_0:=U_{h,0}$, $U^{LT}_m:=U^{LT,h}_m$ and $u^{LT}_m:=u^{LT,h}_m$ for $0\leq m\leq M$.

\begin{proposition}\label{aux_prop}
    Under the assumptions of Definition \ref{aux_process}, the auxiliary process is well-defined and for $0\leq m\leq M-1$ and $t\in[t_m,t_{m+1})$, it satisfies $U^{LT,h}(t)=V^{M,h,1}_m(t)$, where the latter function is defined in (\ref{v_1}) and (\ref{v_1_explicit}). Moreover, the auxiliary process also satisfies $U^{LT,h}(t_m) = U^{LT,h}_m$ for every $0\leq m\leq M$, where $U^{LT,h}_m$ is defined in (\ref{lie_trotter}).
    We also have the corresponding equality in $S_h$, i.e. $u^{LT,h}(t_m)=u^{LT,h}_m$ for each $0\leq m\leq M$.
\end{proposition}
\begin{proof}
    For $m=0$ we trivially have $U^{LT}(t_0)=\exp(0)U_0 = U_0 = U^{LT}_0$. We now proceed by induction. Suppose that for a given $0\leq m\leq M - 1$, we have shown $U^{LT}(t_m)=U^{LT}_m$. Let $t\in[t_m,t_{m+1})$. Then
    \begin{align*}
        U^{LT}(t) &= \exp\left(-t_m\mathcal{M}_L^{-1}\mathcal{S}\right)U_0 + \sum_{k=1}^K\int_0^t \exp\left(-(t_m-l^M(s))\mathcal{M}_L^{-1}\mathcal{S}\right)D_{g\left(U^{LT}\left(l^M(s)\right)\right)} D_{e_k} U^{LT}(s)dB^k_s\\
        &=\exp\left(-t_m\mathcal{M}_L^{-1}\mathcal{S}\right)U_0+\sum_{k=1}^K \int_0^{t_m}\exp\left(-(t_m-l^M(s))\mathcal{M}_L^{-1}\mathcal{S}\right)D_{g\left(U^{LT}\left(l^M(s)\right)\right)} D_{e_k}U^{LT}(s)dB^k_s\\
        &\quad\quad\quad\quad\quad\quad\quad\quad\quad\quad+\sum_{k=1}^K\int_{t_m}^t D_{g\left(U^{LT}\left(t_m\right)\right)} D_{e_k} U^{LT}(s)dB^k_s\\
        &= U^{LT}(t_m) + \sum_{k=1}^K\int_{t_m}^t D_{g\left(U^{LT}_m\right)} D_{e_k} U^{LT}(s)dB^k_s\\
        &=U^{LT}_m + \sum_{k=1}^K\int_{t_m}^t D_{g\left(U^{LT}_m\right)} D_{e_k} U^{LT}(s)dB^k_s.
    \end{align*}
    By the uniqueness of the solution to the equation for $V^{M,h,1}_m$, (\ref{v_1}), we have $U^{LT}(t)=V^{M,h,1}_m(t)$ for $t\in[t_m,t_{m+1})$. This also shows $U^{LT}$ is well-defined on $[t_m,t_{m+1})$. Moreover, since $\tau=t_{m+1}-t_m$, we have
    \begin{align*}
        U^{LT}(t_{m+1}) &= \exp\left(-t_{m+1}\mathcal{M}_L^{-1}\mathcal{S}\right)U_0 + \sum_{k=1}^K\int_0^{t_{m+1}}\exp\left(-(t_{m+1}-l^M(s))\mathcal{M}_L^{-1}\mathcal{S}\right)D_{g\left(U^{LT}\left(l^M(s)\right)\right)} D_{e_k}U^{LT}(s)dB^k_s\\
        &= \exp\left(-\tau\mathcal{M}_L^{-1}\mathcal{S}\right)\bigg(\exp\left(-t_m\mathcal{M}_L^{-1}\mathcal{S}\right)U_0\\
        &\quad\quad\quad\quad\quad\quad\quad\quad\quad\quad\quad+ \sum_{k=1}^K\int_0^{t_m}\exp\left(-(t_m-l^M(s))\mathcal{M}_L^{-1}\mathcal{S}\right)D_{g\left(U^{LT}\left(l^M(s)\right)\right)} D_{e_k}U^{LT}(s)dB^k_s\\
        &\quad\quad\quad\quad\quad\quad\quad\quad\quad\quad\quad\quad+\sum_{k=1}^K\int_{t_m}^{t_{m+1}}D_{g\left(U^{LT}\left(t_m\right)\right)} D_{e_k} U^{LT}(s)dB^k_s\bigg)\\
        &=\exp\left(-\tau\mathcal{M}_L^{-1}\mathcal{S}\right)\left(U^{LT}(t_m) + \sum_{k=1}^K\int_{t_m}^{t_{m+1}}D_{g\left(U^{LT}_m\right)} D_{e_k} U^{LT}(s)dB^k_s\right)\\
        &=\exp\left(-\tau\mathcal{M}_L^{-1}\mathcal{S}\right)\left(U^{LT}_m + \sum_{k=1}^K\int_{t_m}^{t_{m+1}}D_{g\left(U^{LT}_m\right)} D_{e_k} V^{M,h,1}_m(s)dB^k_s\right)\\
        &=\exp\left(-\tau\mathcal{M}_L^{-1}\mathcal{S}\right)V^{M,h,1}_m(t_{m+1}) = U^{LT}_{m+1}.
    \end{align*}
    Hence the proposition is proved.
\end{proof}

Now we show bounds on the second moment of the Lie-Trotter splitting scheme.

\begin{theorem}\label{moment_bounds}
    Under the assumptions of Definition \ref{aux_process}, we have
    \begin{equation}\label{uniform_LT_bound}
        \sup_{t\in[0,T]}\mathbb{E}\left[||u^{LT,h}(t)||_{L^2(\mathcal{D})}^2\right]\leq C,
    \end{equation}
    for some constant $C=C\left(||u_0||_{L^\infty},|\mathcal{D}|,L_f,K_e,T\right)>0$ independent of $h$ and $\tau$.
    In particular
    \begin{equation}\label{discrete_uniform_LT_bound}
        \sup_{0\leq m \leq M}\mathbb{E}\left[||u_m^{LT,h}||_{L^2(\mathcal{D})}^2\right]\leq C,
    \end{equation}
    for the same constant $C=C\left(||u_0||_{L^\infty},|\mathcal{D}|,L_f,K_e,T\right)$.
\end{theorem}
\begin{proof}
    In very similar fashion to proving inequality (\ref{uh_momentbounds}) (just with using the Burkholder-Davis-Gundy-type inequality on a slightly different integrand), we start with the definition of $u^{LT}$ from (\ref{aux_process_equ}) and use the boundedness of $g$ to get
    \begin{align*}
        \mathbb{E}\left[||u^{LT}(t)||_h^2\right]&\leq 2\left(||E_h(l^M(t))I_h u_0||_h^2+\mathbb{E}\left[\left|\left|\sum_{K=1}^K\int_0^t E_h(l^M(t) - l^M(s)) I_h\left(g\left(u^{LT}\left(l^M(s)\right)\right)u^{LT}(s)e_k\right) dB^k_s\right|\right|_h^2\right]\right)\\
        &\leq 2\left(||I_h u_0||_h^2 + \sum_{k=1}^K\int_0^t\mathbb{E}\left[||E_h(l^M(t) - l^M(s)) I_h\left(g\left(u^{LT}\left(l^M(s)\right)\right)u^{LT}(s)e_k\right)||_h^2\right]\text{d}s\right)\\
        &\leq2\left(||u_0||_{L^\infty}^2|\mathcal{D}|+\sum_{k=1}^K\int_0^t\mathbb{E}\left[||I_h\left(g\left(u^{LT}\left(l^M(s)\right)\right)u^{LT}(s)e_k\right)||_h^2\right]\text{d}s\right)\\
        &\leq 2\left(||u_0||_{L^\infty}^2|\mathcal{D}|+L_f^2K_e^2\int_0^t\mathbb{E}\left[||u^{LT}(s)||_h^2\right]\text{d}s\right).
    \end{align*}
    Then by Gr\"onwall's inequality
    \begin{align*}
        \mathbb{E}\left[||u^{LT}(t)||_{h}^2\right]&\leq 2||u_0||_{L^\infty}^2|\mathcal{D}|\left(1+2L_f^2 K_e^2 t\exp(2L_f^2 K_e^2 t)\right).
    \end{align*}
    Using the equivalence of the $h$-norms, (\ref{h_norm_equivalence}), and then taking the supremum yields
    \begin{equation}\label{h_sup_norm_LT}
    \sup_{t\in[0,T]}\mathbb{E}\left[||u^{LT}(t)||_{L^2}^2\right]\leq \sup_{t\in[0,T]}\mathbb{E}\left[||u^{LT}(t)||_{h}^2\right]\leq C,
    \end{equation}
    where
    \[
        C=C\left(||u_0||_{L^\infty},|\mathcal{D}|,L_f,K_e,T\right)=2||u_0||_{L^\infty}^2|\mathcal{D}|\left(1+2L_f^2 K_e^2 T\exp(2L_f^2 K_e^2 T)\right).
    \]
    Now, due to Proposition \ref{aux_prop}, in particular $u^{LT}(t_m)=u_m^{LT}$ for $0\leq m\leq M$, inequality (\ref{discrete_uniform_LT_bound}) also holds.
\end{proof}

To conclude the analysis of the auxiliary process, we prove partial time regularity.
\begin{lemma}\label{aux_reg}
    Under the assumptions of Definition \ref{aux_process}, for every $t\in[0,T]$ we have
    \begin{equation}\label{partial_time_regularity}
        \mathbb{E}\left[||u^{LT,h}(t)-u^{LT,h}(l^M(t))||_h^2\right]^\frac{1}{2}\leq C\left|t-l^M(t)\right|^\frac{1}{2} \leq C \tau^{1/2}.
    \end{equation}
    for some constant $C=C\left(||u_0||_{L^\infty},|\mathcal{D}|,L_f,K_e,T\right)>0$ independent of $h$ and $\tau$.
    \begin{proof}
In very similar fashion to proving inequality (\ref{uh_momentbounds}) (just with using the Burkholder-Davis-Gundy-type inequality on a slightly different integrand), we again start with the definition of $u^{LT}$ from (\ref{aux_process_equ}), then use the boundedness of $g$ and the uniform bound, (\ref{h_sup_norm_LT}), to get
    \begin{align*}
    \mathbb{E}\left[||u^{LT}(t)-u^{LT}(l^M(t))||_h^2\right]&=\mathbb{E}\left[\left|\left|\sum_{k=1}^K\int_{l^M(t)}^t I_h\left(g(u^{LT}(l^M(t)))u^{LT}(s)e_k\right)\text{d}B^k_s\right|\right|_h^2\right]\\
    &\leq\sum_{k=1}^K\int_{l^M(t)}^t\mathbb{E}\left[\left|\left|I_h\left(g(u^{LT}(l^M(t)))u^{LT}(s)e_k\right)\right|\right|_h^2\right]\text{d}s\\
    &\leq L_f^2K_e^2\int_{l^M(t)}^t\mathbb{E}\left[||u^{LT}(s)||_h^2\right]\text{d}s\\
    &\leq C\left|t-l^M(t)\right|,
    \end{align*}
    where
    \[
    C= 2L_f^2K_e^2||u_0||_{L^\infty}^2|\mathcal{D}|\left(1+2L_f^2 K_e^2 T\exp(2L_f^2 K_e^2 T)\right).
    \]
    \end{proof}
\end{lemma}

\subsection{Convergence}

To prove convergence of the fully discrete solution, (\ref{lie_trotter}), to the semi-discrete solution, (\ref{mild_sh}), we make use of the following discrete Gr\"onwall inequality,  see \cite[Lemma A.3]{Kruse}.
\begin{lemma}\label{gronwall}
    Let $c\geq 0$ and let $(\psi_j)_{j\geq1}$ and $(v_j)_{j\geq1}$ be nonnegative sequences. If
    \[
    \psi_j\leq c + \sum_{i=1}^{j-1}v_i\psi_i\quad\quad\text{for}\quad\quad j\geq1,
    \]
    then
    \[
    \psi_j\leq c\exp\left(\sum_{i=1}^{j-1}v_i\right)\quad\text{for}\quad j\geq1.
    \]
\end{lemma}

The proof of the following theorem is adapted from  \cite[Theorem 6]{BCU_1}.

\begin{theorem}\label{convergence}
    Let Assumptions \ref{assumption_1} and \ref{assumption_2} hold. Let $M\in\mathbb{N}$ and define $\tau = \frac{T}{M}$ and $t_m:=m\tau$ for $m=0,...,M$. Then for every 
    $\varepsilon >0$
    \begin{equation}\label{strongstrong}
        \sup_{0\leq m\leq M}\mathbb{E}\left[||u_h(t_m)-u^{LT,h}_m||_{L^2(\mathcal{D})}^2\right]^\frac{1}{2}\leq C\tau^{\frac{1}{2}-\varepsilon},
    \end{equation}
    for some constant $C=C\left(\varepsilon,C_1, d, \text{diam}(\mathcal{D}), \partial\mathcal{D},|\mathcal{D}|,||u_0||_{L^\infty},||u_0||_{H^2},L_f,K_e,T\right)>0$ independent of $h$ and $\tau$.\par
    Moreover, we have
    \begin{equation}\label{stronger}
        \sup_{t\in[0,T]}\mathbb{E}\left[\left|\left|u_h(t)-u^{LT,h}_{m(t)}\right|\right|_{L^2(\mathcal{D})}^2\right]^\frac{1}{2}\leq C\tau^{\frac{1}{2}-\varepsilon},
    \end{equation}
    where $m(t)=\frac{l^M(t)}{\tau}$ (it returns the index $m$ of the interval $[t_m, t_{m+1})$ that contains the point t) and $C$ is some constant $C=C\left(\varepsilon,C_1, d, \text{diam}(\mathcal{D}), \partial\mathcal{D},|\mathcal{D}|,||u_0||_{L^\infty},||u_0||_{H^2},L_f,K_e,T\right)>0$ independent of $h$ and $\tau$.
    \begin{proof}
        For notational convenience, for each $0\leq m\leq M$, we denote the error and its norm by
        \begin{align*}
            E_m &:=u_h(t_m)-u^{LT}_m,\\
            \mathcal{E}_m &:= \mathbb{E}\left[||E_m||_{h}^2\right].
        \end{align*}
        Using the mild solution of $u_h$, (\ref{mild_sh}), Definition \ref{aux_process} and Proposition \ref{aux_prop}, we may write
        \begin{align*}
            E_m &= u_h(t_m)-u^{LT}(t_m)\\
            &=\sum_{k=1}^K\int_0^{t_m}E_h(t_m-s)I_h(f(u_h(s))e_k)dB^k_s\\
            &\quad\quad\quad\quad\quad-\sum_{k=1}^K\int_0^{t_m}E_h(t_m-l^M(s))I_h(g\left(u^{LT}\left(l^M(s)\right)\right)u^{LT}(s)e_k)dB^k_s\\
            &=E^{(1)}_m + E^{(2)}_m,
        \end{align*}
        where
    \begin{align*}
            E^{(1)}_m &:= \sum_{k=1}^K\int_0^{t_m}E_h(t_m-s)I_h\left(\left(f(u_h(s))-g(u^{LT}(l^M(s)))u^{LT}(s)\right)e_k\right)dB^k_s,\\
            E^{(2)}_m &:= \sum_{k=1}^K\int_0^{t_m}\left(E_h(t_m-s)-E_h(t_m-l^M(s))\right)I_h\left(g(u^{LT}(l^M(s)))u^{LT}(s) e_k\right)dB^k_s.
        \end{align*}
        We further decompose $E^{(1)}_m$ as $E^{(1)}_m=E^{(1,1)}_m+E^{(1,2)}_m+E^{(1,3)}_m$, where
        \begin{align*}
            E^{(1,1)}_m &:= \sum_{k=1}^K\int_0^{t_m}E_h(t_m-s)I_h\left(\left(f(u_h(s))-f(u_h(l^M(s)))\right)e_k\right)dB^k_s,\\
            E^{(1,2)}_m &:= \sum_{k=1}^K\int_0^{t_m}E_h(t_m-s)I_h\left(\left(f(u_h(l^M(s)))-f(u^{LT}(l^M(s)))\right)e_k\right)dB^k_s,\\
            E^{(1,3)}_m &:= \sum_{k=1}^K\int_0^{t_m}E_h(t_m-s)I_h\left(\left(f(u^{LT}(l^M(s)))-g(u^{LT}(l^M(s)))u^{LT}(s)\right)e_k\right)dB^k_s.
        \end{align*}
        We now look at bounding the $h$-norms of each of the terms.\par
        For $E^{(1,1)}_m$, we use the Burkholder-Davis-Gundy-type inequality for Hilbert space valued stochastic integrals \cite[Proposition 2.12]{Kruse} applied to $(s,\omega)\mapsto E_h(t_m-s)I_h\left(\left(f(u_h(s))-f(u_h(l^M(s)))\right)\cdot\right)$, the contractive properties of the semi-group, Lipschitz continuity of $f$ and the $\frac{1}{2}$-H\"older continuity of $u_h$ (Theorem \ref{uh_properties}) to get
        \begin{equation}\label{E11}
            \begin{split}
            \mathbb{E}\left[||E^{(1,1)}_m||_{h}^2\right]&\leq \sum_{k=1}^K\int_0^{t_m}\mathbb{E}\left[\left|\left|E_h(t_m-s)I_h\left(\left(f(u_h(s))-f(u_h(l^M(s)))\right)e_k\right)\right|\right|_h^2\right]\text{d}s\\
            &\leq L_f^2K_e^2\int_0^{t_m}\mathbb{E}\left[||u_h(s)-u_h(l^M(s))||_{h}^2\right]\text{d}s\\
            &\leq L_f^2K_e^2C\int_0^{t_m}|s-l^M(s)|\text{d}s\\
            &\leq L_f^2K_e^2TC\tau,
            \end{split}
        \end{equation}
        where $C=C\left(C_1, d, \text{diam}(\mathcal{D}), \partial\mathcal{D},|\mathcal{D}|,||u_0||_{L^\infty},||u_0||_{H^2},L_f,K_e,T\right)$ is from (\ref{uh_temporal}). 
        For $E^{(1,2)}_m$ we use similar reasoning to get
        \begin{equation}\label{E12}
        \begin{split}
            \mathbb{E}\left[||E^{(1,2)}_m||_{h}^2\right]&\leq \sum_{k=1}^K\int_0^{t_m}\mathbb{E}\left[\left|\left|E_h(t_m-s)I_h\left(\left(f(u_h(l^M(s)))-f(u^{LT}(l^M(s)))\right)e_k\right)\right|\right|_h^2\right]\text{d}s\\
            &\leq L_f^2K_e^2\int_0^{t_m}\mathbb{E}\left[||u_h(l^M(s))-u^{LT}(l^M(s))||_{h}^2\right]\text{d}s\\
            &=L_f^2K_e^2\sum_{j=0}^{m-1}\int_{t_j}^{t_{j+1}}\mathbb{E}\left[||u_h(t_j)-u^{LT}(t_j)||_{h}^2\right]\text{d}s\\
            &=L_f^2K_e^2\tau\sum_{j=0}^{m-1}\mathcal{E}_j.
            \end{split}
        \end{equation}
        For $E^{(1,3)}_m$ we use the same reasoning as above, Lemma \ref{aux_reg} and the fact $f(r)=g(r)r$ for every $r\in\mathbb{R}$ to get 
        \begin{equation}\label{E13}
        \begin{split}
            \mathbb{E}\left[||E^{(1,3)}_m||_{h}^2\right]&\leq \sum_{k=1}^K\int_0^{t_m}\mathbb{E}\left[\left|\left|E_h(t_m-s)I_h\left(g(u^{LT}(l^M(s)))\left(u^{LT}(l^M(s))-u^{LT}(s)\right)e_k\right)\right|\right|_h^2\right]\text{d}s\\
            &\leq L_f^2K_e^2\int_0^{t_m}\mathbb{E}\left[||u^{LT}(l^M(s))-u^{LT}(s)||_h^2\right]\text{d}s\\
            &\leq L_f^2K_e^2C\int_0^{t_m}\left|s-l^M(s)\right|\text{d}s\\
            &\leq L_f^2K_e^2TC \tau,
            \end{split}
        \end{equation}
        where $C=C\left(||u_0||_{L^\infty},|\mathcal{D}|,L_f,K_e,T\right)$ is from (\ref{partial_time_regularity}).\par
        Let $\alpha\in\left[0,\frac{1}{2}\right)$ be arbitrary. 
        For $E^{(2)}_m$ we again use the Burkholder-Davis-Gundy-type inequality for Hilbert space valued stochastic integrals, then (\ref{semigroup_lemma_1}), inequalities (\ref{semigroup_lemma_3}) and (\ref{semigroup_lemma_2}), and finally the boundedness of $g$ and the uniform bound (\ref{h_sup_norm_LT}) to get
        \begin{equation}\label{E2}
        \begin{split}
            \mathbb{E}\left[||E^{(2)}_m||_{h}^2\right]&\leq \sum_{k=1}^K\int_0^{t_m}\mathbb{E}\left[\left|\left|\left(E_h(t_m-s)-E_h(t_m-l^M(s))\right)I_h\left(g(u^{LT}(l^M(s)))u^{LT}(s) e_k\right)\right|\right|_h^2\right]\text{d}s\\
            &= \sum_{k=1}^K\int_0^{t_m}\mathbb{E}\left[\left|\left|(-\Delta_h)^{-\alpha}\left(\operatorname{Id}_{S_h}-E_h(s-l^M(s))\right)(-\Delta_h)^\alpha E_h(t_m-s)I_h\left(g(u^{LT}(l^M(s)))u^{LT}(s) e_k\right)\right|\right|_h^2\right]\text{d}s\\
            &\leq C_4^2 C_3^2\sum_{k=1}^K\int_0^{t_m}\left|s-l^M(s)\right|^{2\alpha}\frac{1}{(t_m-s)^{2\alpha}}\mathbb{E}\left[\left|\left|I_h\left(g(u^{LT}(l^M(s)))u^{LT}(s) e_k\right)\right|\right|_h^2\right]\text{d}s\\
            &\leq C_4^2 C_3^2L_f^2K_e^2\tau^{2\alpha}\int_0^{t_m}\frac{1}{(t_m-s)^{2\alpha}}\mathbb{E}\left[\left|\left|u^{LT}(s)\right|\right|_h^2\right]\text{d}s\\
            &\leq C_4\left(\alpha\right)^2C_3(\alpha)^2L_f^2K_e^2 C\frac{T^{1-2\alpha}}{1-2\alpha}\tau^{2\alpha},
            \end{split}
        \end{equation}
        where $C=C\left(||u_0||_{L^\infty},|\mathcal{D}|,L_f,K_e,T\right)$ is from (\ref{uniform_LT_bound}).
        Using the fact that
        \begin{equation}\label{tau_alpha}
            \tau\leq T^{1-2\alpha}\tau^{2\alpha}\leq\max\{T,1\}\tau^{2\alpha}
        \end{equation}
        and combining inequalities (\ref{E11}), (\ref{E12}), (\ref{E13}) and (\ref{E2}) yields
        \begin{equation}\label{gronwall_combiner}
            \mathcal{E}_m \leq4\left(\mathbb{E}\left[||E_m^{(1,1)}||_{h}^2\right]+\mathbb{E}\left[||E_m^{(1,2)}||_{h}^2\right]+\mathbb{E}\left[||E_m^{(1,3)}||_{h}^2\right]+\mathbb{E}\left[||E_m^{(2)}||_{h}^2\right]\right)
            \leq C\tau^{2\alpha}+L_f^2K_e^2\tau\sum_{j=0}^{m-1}\mathcal{E}_j,
        \end{equation}
        where $C=C\left(\alpha, C_1, d, \text{diam}(\mathcal{D}), \partial\mathcal{D},|\mathcal{D}|,||u_0||_{L^\infty},||u_0||_{H^2},L_f,K_e,T\right)$ can be inferred from (\ref{E11}), (\ref{E13}), (\ref{E2}) and (\ref{tau_alpha}).
        Using the discrete Gr\"onwall inequality (Lemma \ref{gronwall}) and the fact that $M\tau=T$, gives us
        \[
        \sup_{0\leq m\leq M}\mathcal{E}_m \leq C\exp(T L_f^2K_e^2)\tau^{2\alpha},
        \]
        where $C=C\left(\alpha, C_1, d, \text{diam}(\mathcal{D}), \partial\mathcal{D},|\mathcal{D}|,||u_0||_{L^\infty},||u_0||_{H^2},L_f,K_e,T\right)$ is from (\ref{gronwall_combiner}).
        Using the equivalence of $h$-norms, (\ref{h_norm_equivalence}), then yields inequality (\ref{strongstrong}).\par
        Now let $t\in[0,T]$. Then by the triangle inequality, the discrete time error (\ref{strongstrong}) and the time regularity of $u_h$ (Theorem \ref{uh_properties}), we have
        \begin{align*}
        \mathbb{E}\left[\left|\left|u_h(t)-u^{LT,h}_{m(t)}\right|\right|_{L^2}^2\right]^\frac{1}{2}&\leq\mathbb{E}\left[\left|\left|u_h(t)-u_h(l^M(t))\right|\right|_{L^2}^2\right]^\frac{1}{2}+\mathbb{E}\left[\left|\left|u_h(l^M(t))-u^{LT,h}_{m(t)}\right|\right|_{L^2}^2\right]^\frac{1}{2}\\
        &\leq C|t-l^M(t)|^\frac{1}{2}+C\tau^{\alpha}\leq C\tau^\frac{1}{2}+C\tau^{\alpha}\\
        &\leq C\tau^{\alpha}
        \end{align*}
        where $C=C\left(\alpha, C_1, d, \text{diam}(\mathcal{D}), \partial\mathcal{D},|\mathcal{D}|,||u_0||_{L^\infty},||u_0||_{H^2},L_f,K_e,T\right)$ can be inferred from (\ref{uh_temporal}), (\ref{strongstrong}) and (\ref{tau_alpha}).
    \end{proof}
\end{theorem}

Now, we may use Theorem \ref{convergence} and Theorem \ref{strong_error_semi} to prove convergence of  the Lie-Trotter splitting scheme, (\ref{lie_trotter}), to the solution $u$ of (\ref{main_equation}).\par
It is important to note that the convergence rate with respect to the spatial discretization is likely not optimal. It is expected the rate will be $h^2$ or $h^{2-}$, which is supported by the numerical experiments in the next section. To improve our result to this rate, an inequality like
\[
\sup_{t\in[0,T]}\mathbb{E}\left[\left|\left|u(t)-u_h(t)\right|\right|_{L^2(\mathcal{D})}^2\right]^\frac{1}{2}\leq Ch^2,
\]
for some constant $C$ independent of $h$, would have to be proven. Obtaining this result would require a nontrivial modification of the arguments used in \cite{Djurdjevac}. While the corresponding estimate is classical and well understood in the deterministic setting, we are not aware of an analogous result available in the stochastic framework considered here. For these reasons, establishing the optimal spatial convergence rate in the stochastic case is deferred to future work.
\begin{corollary}\label{main_result}
    Under the same assumptions as in Theorem \ref{convergence} and under the additional Assumption \ref{assumption_3}, for every $\varepsilon>0$ there exists a constant $C$, independent of $h$ and $\tau$ such that
    \[
        \sup_{t\in[0,T]}\mathbb{E}\left[\left|\left|u(t)-u^{LT,h}_{m(t)}\right|\right|_{L^2(\mathcal{D})}^2\right]^\frac{1}{2}\leq C\left(h+\tau^{\frac{1}{2}-\varepsilon}\right),
    \]
    where $m(t)=\frac{l^M(t)}{\tau}$ (it returns the index $m$ of the interval $[t_m, t_{m+1})$ that contains the point t).
\end{corollary}
\begin{proof}
    Let $\alpha\in\left[0,\frac{1}{2}\right)$ and $t\in[0,T]$. Then by Theorems \ref{convergence} and \ref{strong_error_semi} we have
    \begin{align*}
        \mathbb{E}\left[\left|\left|u(t)-u^{LT}_{m(t)}\right|\right|_{L^2}^2\right]^\frac{1}{2}&\leq\mathbb{E}\left[\left|\left|u(t)-u_h(t)\right|\right|_{L^2}^2\right]^\frac{1}{2}+\mathbb{E}\left[\left|\left|u_h(t)-u^{LT}_{m(t)}\right|\right|_{L^2}^2\right]^\frac{1}{2}\\
        &\leq C_6\:h+C\tau^\alpha\\
        &\leq C\left(h+\tau^\alpha\right),
    \end{align*}
    Where $C=C\left(\alpha, u, C_1, d, \text{diam}(\mathcal{D}), \partial\mathcal{D},|\mathcal{D}|,||u_0||_{L^\infty},||u_0||_{H^2},L_f,K_e,T\right)$ can be inferred from (\ref{stronger}) and (\ref{trve_solution_error}).
\end{proof}
\section{Numerical Experiments}\label{section6}

Numerical experiments are now presented to illustrate the theoretical results and to explore scenarios beyond the scope of the analysis. In particular, we simulate the strong error
\begin{equation}\label{strong_error}
\sup_{t\in[0,T]}\mathbb{E}\left[\left|\left|u(t)-u^{LT,h,\tau}_{m(t)}\right|\right|_{L^2}^2\right]^\frac{1}{2}
\end{equation}
for various temporal and spatial discretization parameters to show the proven temporal regularity and an improved spatial regularity, with equation parameters and triangulations satisfying Assumptions \ref{assumption_1} and \ref{assumption_2} respectively. We write $u^{LT,h,\tau}$ with the added $\tau$ superscript to emphasise the dependence on the time discretization parameter. We note that simulations with a true nonlinearity, $f$, are a novelty with respect to the simulations performed in \cite{Djurdjevac}, where a linear $f(u)=\lambda u$ was used.
Moreover, we consider the weak error
\begin{equation}\label{weak_error}
    \left|\mathbb{E}\left[\phi(u(T))\right] - \mathbb{E}\left[\phi(u^{LT,h,\tau}(T))\right]\right|,
\end{equation}
for various time and space discretization parameters  with an appropriate functional $\phi:L^2(\mathcal{D})\to\mathbb{R}$
and discuss the results.\par
We perform our numerical experiments on the $2$-dimensional spatial domain $\mathcal{D} = (0,1)\times(0,1)\subset\mathbb{R}^2$ and time domain $\left[0,\frac{1}{2}\right]$ ($T=\frac{1}{2}$).  The triangulation on $\mathcal{D}$ is defined by dividing each side by $N$, where $N$ will vary in the experiments, to form a uniform grid with each square in the grid being divided into two triangles. With this triangulation, Assumption \ref{assumption_1} (a) and Assumption \ref{assumption_2} are satisfied with $C_1$ independent of $N$.\par
The orthonormal functions on $\mathcal{D}$, $e_k$, are the sinusoidal basis functions
\[
e_{i,j}(x,y)=2\sin(\pi i x)\sin(\pi j y)\quad\text{for}\quad i,j\in\mathbb{N}.
\]
We will only consider noise dimension $K=n^2$ for some $n\in\mathbb{N}$, so that the basis functions used are $e_{i,j}$ for $1\leq i,j\leq n$.\par
The globally Lipschitz $C^1$ nonlinearity we consider is an approximation of the square root function given by
\begin{equation}\label{sqrt_approx}
f_\delta(x) = \begin{cases}
    \frac{1}{\sqrt{\delta}}x, & |x|\leq\frac{\delta}{2},\\
    -\frac{2\sqrt{\delta}}{\delta^3}x^3+\text{sign}(x)\frac{4}{\delta\sqrt{\delta}}x^2-\frac{3}{2\sqrt{\delta}}x+\text{sign}(x)\frac{\sqrt{\delta}}{2}, & \frac{\delta}{2}\leq|x|\leq\delta,\\
    \text{sign}(x)\sqrt{|x|}, & |x|\geq\delta.
\end{cases}
\end{equation}
For details, see \cite[Example 1]{DjurdKrempPerk}. We consider the initial condition $u_0(x,y)=\sin(\pi x)\sin(\pi y)$.\par
Regarding nonnegativity preservation, recall from the discussion preceding Theorem \ref{splitting_positivity} that the only component of the numerical method (\ref{lie_trotter}) requiring verification is the matrix $\exp\left(-\tau\mathcal{M}_L^{-1}\mathcal{S}\right)$. For all the performed experiments, these matrices were computed and confirmed to be element-wise nonnegative, ensuring that the nonnegativity property is indeed satisfied in practice. For more details about nonnegativity preservation and comparisons to other numerical schemes, see \cite{Djurdjevac}. \par
The code used for the numerical experiments is publicly available on GitHub at:\par
https://github.com/redraeho/Code-for-the-paper.

\subsection{Strong Error}

To illustrate the proven temporal convergence rate in Corollary \ref{main_result} and an improved spatial convergence rate, we run simulations with $K=4$ and use the nonlinearity from (\ref{sqrt_approx}) with $\delta=0.1$. We aim to estimate the strong error (\ref{strong_error}).
To this end, we estimate $u$ with a reference solution, $u_{ref}$, which is given by the splitting scheme, (\ref{lie_trotter}), for $M=M_{max}:=2^{12}$ (corresponding to $\tau=\tau_{min}:=2^{-13}$) and $N=N_{max}:=2^6$ (corresponding to $h=h_{min}:= \sqrt{2}\cdot2^{-6}$). We estimate the expectation with $150$ realizations of Brownian paths (corresponding to $(\omega_j)_{j=1}^{150}\subset\Omega$) and calculate the $L^2$ norms using Lemma \ref{norm_equivalence}. Finally, we approximate the strong error \eqref{strong_error} by the Monte--Carlo
estimator
\begin{equation}\label{numerical_simulation}
    \sup_{0\leq m\leq M}\frac{1}{\sqrt{150}}\sqrt{\sum_{j=1}^{150}\left|\left|u^{LT,h,\tau}_m(\omega_j)-u_{ref}(t_m,\omega_j)\right|\right|_{L^2}^2}
\end{equation}
for fixed $N=N_{max}$ (or $h=h_{min}$) and the values $M=2^3,2^4,...,2^{11}$ (corresponding to $\tau=2^{-4},2^{-5},...,2^{-12}$); and separately for fixed $M=M_{max}$ (or $\tau=\tau_{min}$) and the values $N=2^2, 2^3, 2^4, 2^5$ (corresponding to $h=\sqrt{2}\cdot2^{-2},\sqrt{2}\cdot2^{-3},\sqrt{2}\cdot2^{-4},\sqrt{2}\cdot2^{-5}$). \par
The results for varying the time discretization parameter are shown in
Figure~\ref{fig:strong_error_tau} on a log--log scale, while the results for
varying the spatial discretization parameter are shown in
Figure~\ref{fig:strong_error_h}, also on a log--log scale. The number of
realizations was chosen sufficiently large so that the associated standard
errors - computed as the sample standard deviation divided by the square root of
the number of realizations - were negligible. For example, the largest observed
standard error occurred for $h=h_{\min}$ and $\tau=2^{-7}$, where the empirical
mean in \eqref{strong_error} was $5.3\times 10^{-4}$ with a standard error of
$8.7\times 10^{-5}$.

\begin{figure}[ht]
  \begin{subfigure}[b]{0.5\textwidth}
    \includegraphics[width=\textwidth]{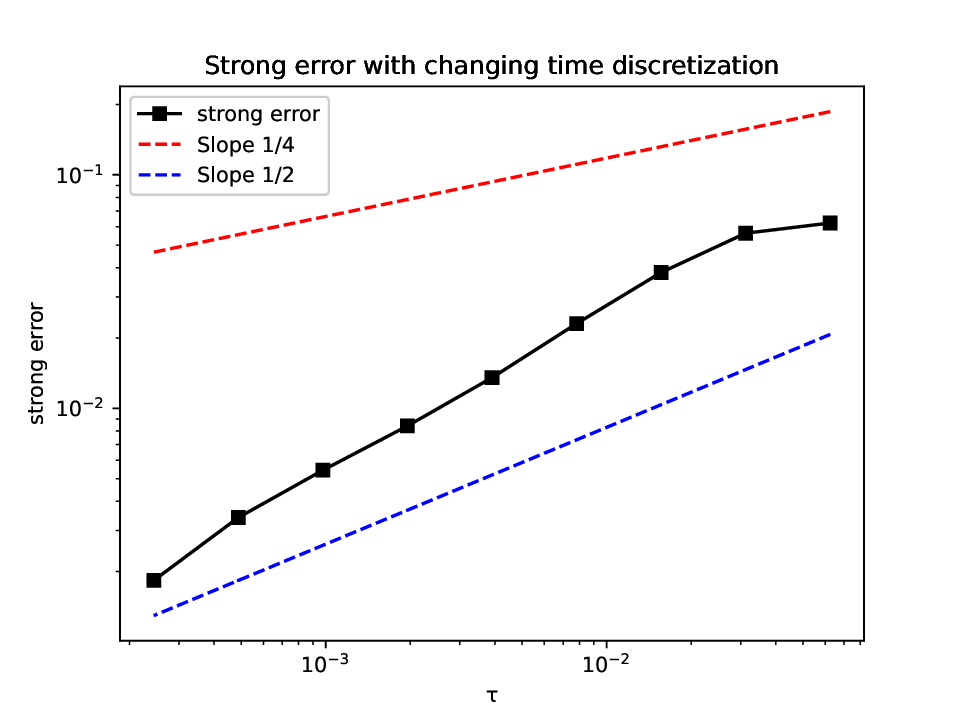}
    \caption{Changing $\tau$ with $h=\sqrt{2}\cdot2^{-6}$.}
    \label{fig:strong_error_tau}
  \end{subfigure}
  \hfill
  \begin{subfigure}[b]{0.5\textwidth}
    \includegraphics[width=\textwidth]{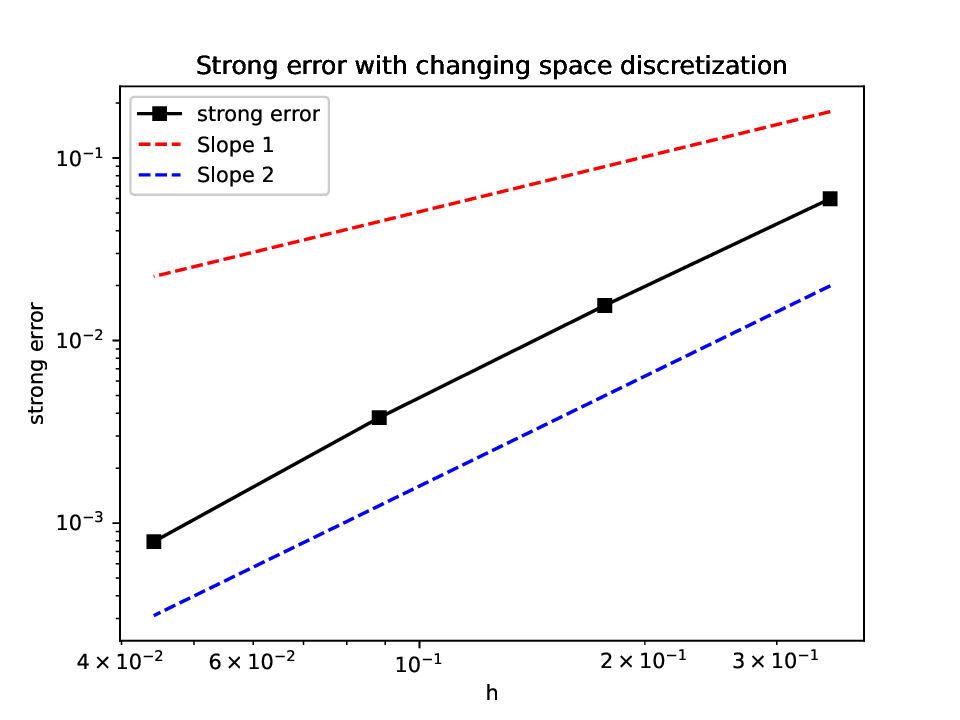}
    \caption{Changing $h$ with $\tau=2^{-13}$.}
    \label{fig:strong_error_h}
  \end{subfigure}
  \caption{Plots for the strong error, (\ref{strong_error}), when the nonlinearity is given by (\ref{sqrt_approx}) with $\delta=0.1$. Both axes are scaled logarithmically.}
  \label{fig:strong_error}
\end{figure}

As is suggested by Figure \ref{fig:strong_error},  the strong error, (\ref{strong_error}), has convergence rates close to $\tau^\frac{1}{2}$ and $h^2$.\par

It is also important to assess the strong error relative to the magnitude of the
reference solution in the $L^2$-norm. In Figure~\ref{fig:strong_error_tau}, the
relative error remained below $0.1$ for all but the two largest values of
$\tau$. Similarly, in Figure~\ref{fig:strong_error_h}, the relative error was
approximately $0.12$ for the largest value of $h$ and below $0.03$ for
smaller spatial step sizes. These observations confirm that the reported absolute
 errors are representative of the actual approximation quality and justify focusing on absolute errors in
the convergence study.

Since the constant in Corollary \ref{main_result} depends on $K_e$ (which equals $2\sqrt{K}$ in the present setting), we also investigated the effect of increasing the noise dimension $K$.
The same experiment as before was repeated with $K=1024$. The resulting plots did not exhibit a clear convergence behavior,
which is understandable as there is a lot more noise. This suggests that much finer discretization parameters would likely be required to observe the expected convergence behavior.

The constant in Corollary \ref{main_result} also depends on the Lipschitz constant $L_f$. To assess the importance of this dependence, an additional experiment was carried out where the nonlinearity $f$ was not Lipschitz. We repeated the earlier test with $K=4$ and the nonlinearity $f(x)=\sqrt{\max\{0,x\}}$, leading to the corresponding function $g(x)=1/\sqrt{x}$ for $x>0$ and $g(x)=0$ for $x\leq0$, which is singular near the origin. The computation occasionally produced numerical overflow, and almost always led to zero-everywhere solutions after a short amount of time. Despite this, the graphs of the strong error squared (Figure \ref{fig:strong_error_sqrt}) suggest the same convergence rates as in Figure \ref{fig:strong_error}. We plot the strong error squared in Figure \ref{fig:strong_error_sqrt} so we can also show the not insignificant standard error of the empirical mean approximating the expectation. If the numerical solutions in this case do indeed converge to the true solution $u$, then, even with the standard error present, Figure \ref{fig:strong_error_sqrt} suggests the convergence rates would be close to $\tau^\frac{1}{2}$ and $h^2$.
The relative errors in Figure~\ref{fig:strong_error_sqrt_h} exhibit similar
behavior to those in Figure~\ref{fig:strong_error_h}. In contrast, for
Figure~\ref{fig:strong_error_sqrt_tau}, the relative error falls below \(0.1\)
only for the smallest value of \(\tau\). For the remaining cases, the five smallest
values of \(\tau\) yield relative errors below \(0.2\), while all other cases
remain below \(0.4\).

\begin{figure}[!htb]
  \begin{subfigure}[b]{0.5\textwidth}
    \includegraphics[width=\textwidth]{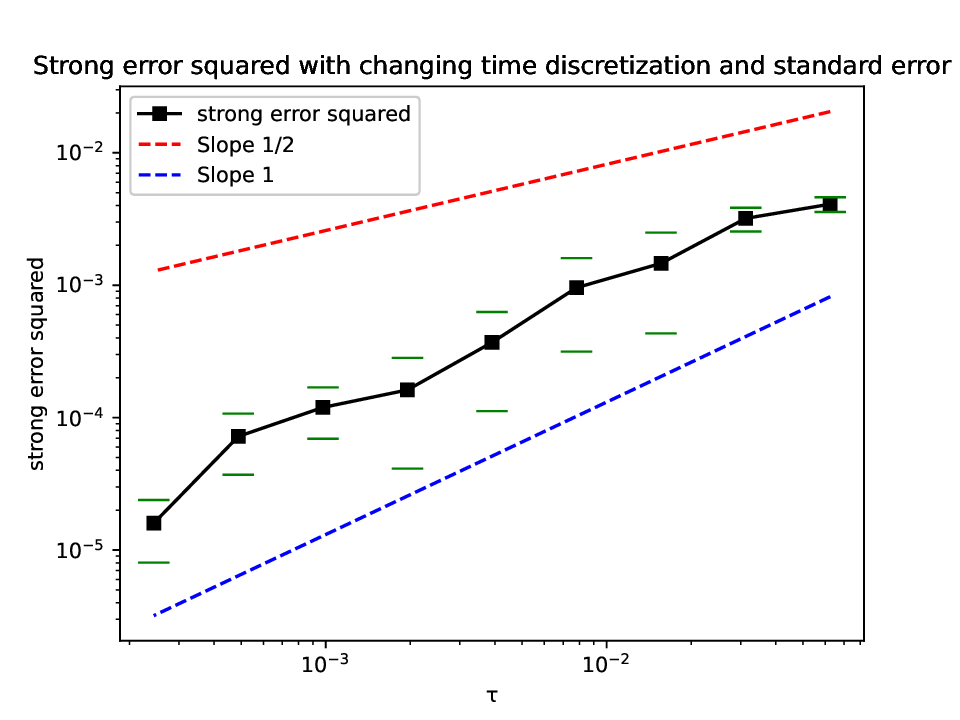}
    \caption{Changing $\tau$ with $h=\sqrt{2}\cdot2^{-6}$.}
    \label{fig:strong_error_sqrt_tau}
  \end{subfigure}
  \hfill
  \begin{subfigure}[b]{0.5\textwidth}
    \includegraphics[width=\textwidth]{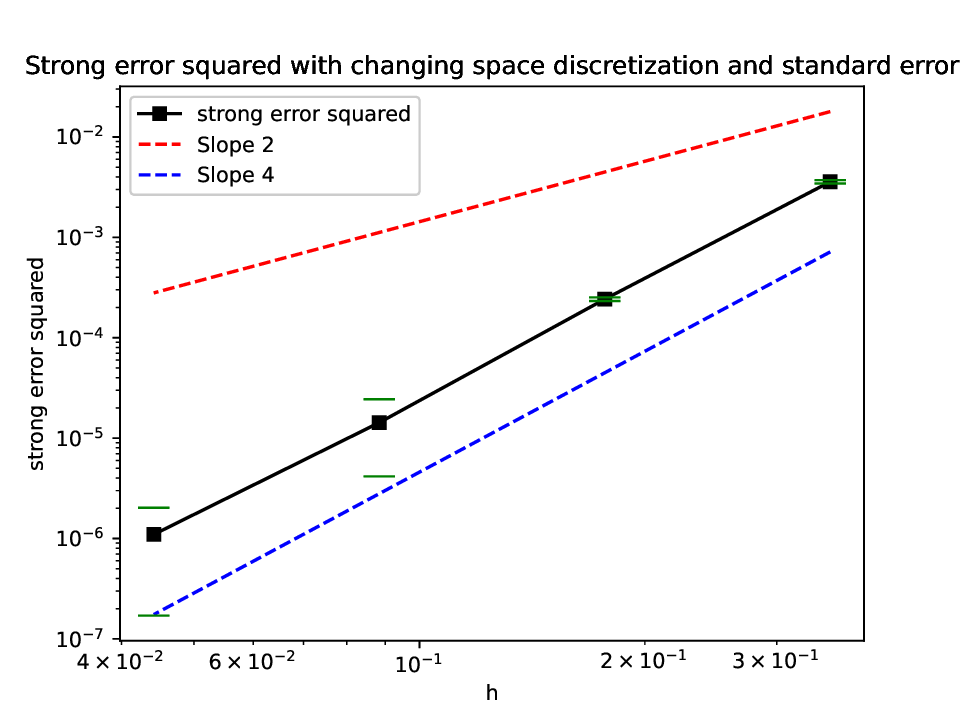}
    \caption{Changing $h$ with $\tau=2^{-13}$.}
    \label{fig:strong_error_sqrt_h}
  \end{subfigure}
  \caption{Plots for the strong error squared when the nonlinearity is $f(x)=\sqrt{\max\{0,x\}}$. The standard error is represented by the green bars and both axes are scaled logarithmically.}
  \label{fig:strong_error_sqrt}
\end{figure}

\subsection{Weak Error}

As outlined in the introduction, a natural direction for future work is the investigation of convergence rates for the weak error (\ref{weak_error}). We consider $\phi=||\cdot||_{L^2(\mathcal{D})}^2$ and the parameters outlined for (\ref{numerical_simulation}) were used to calculate
\[
  \frac{1}{150}  \left|\sum_{j=1}^{150}||u^{LT,h,\tau}_M(\omega_j)||_{L^2}^2-||u_{ref}(T,\omega_j)||_{L^2}^2\right|
\]
for the same varying $M$ and $N$. Since weak error rate simulations are prone to large statistical errors, the same Brownian paths were used to estimate the expectations in (\ref{weak_error}) to reduce variance.
Despite this, the weak errors still produced significant standard errors.\par
The simulations did show a decay in the weak error, but convergence rates could not be inferred.
The weak errors ranged from $3\times 10^{-6}$ to $1\times 10^{-7}$, with the relative errors  being less than $0.1$ only for $\tau=2^{-3}, 2^{-5}, 2^{-8}, 2^{-9}$. Excluding the two largest $\tau$ values, the remaining relative errors were less than $0.32$.\par
In the hope of reducing the standard error, the number of realizations used to estimate the weak error was increased. We ran the same experiment, but this time for fixed $N=2^4$ and $10000$ realizations of the Brownian paths, with the same time discretization parameters. This did not reduce the standard error significantly and nothing more could be deduced. More elaborate variance reduction techniques or a much higher number of realizations would be needed to lower the standard error.
Furthermore, it is conceivable that a Strang splitting scheme could offer a reduction in the weak error, and we note that the strong error simulation analysis for such a scheme in the linear case was already carried out in \cite{Djurdjevac}.
\par
As a last set of experiments, we also considered the autocorrelation of the solution, a quantity that is often of interest in applications. The results, however, were not informative, which is consistent with the difficulties already encountered in simulating the weak error. Nevertheless, understanding autocorrelation behaviour is an interesting direction for future work.
Possible promising directions are spectral or hybrid discretizations that preserve long-time statistical structure, multilevel Monte Carlo for correlation observables or reduced-order modelling based on linearised dynamics.
\section*{Acknowledgements}

This research has been (partially) supported by Deutsche Forschungsgemeinschaft (DFG) through the research grant CRC 1114 "Scaling Cascades in Complex Systems", Project Number 235221301, Project C10.
ADj gratefully acknowledges funding by the Daimler-Benz Foundation as part of the scholarship program for junior professors and postdoctoral researchers. OH is grateful for the support from the Berlin Mathematical School (BMS), which is funded by the Deutsche Forschungsgemeinschaft (DFG, German Research Foundation) under Germany's Excellence Strategy - The Berlin Mathematics Research Center MATH+ (EXC-2046/1, project ID: 390685689), and the funding by the Deutsche Forschungsgemeinschaft (DFG, German Research Foundation) - CRC/TRR 388 "Rough Analysis, Stochastic Dynamics and Related Fields“ - Project ID 516748464, Project B09. This research was initiated while CLB was visiting the Berlin Mathematics Research Center MATH+.

\bibliographystyle{plain}
\bibliography{8_Bibliography}

\end{document}